\documentclass[fleqn]{article}
\usepackage[utf8]{inputenc}
\usepackage{enumerate,amsthm,amssymb,xcolor,hyperref,comment,framed,tikz,ogonek}
\usepackage[margin=1in]{geometry}

\usepackage[leqno]{amsmath}

\newenvironment{packed_enum}{
\begin{enumerate}
  \setlength{\itemsep}{1pt}
  \setlength{\parskip}{0pt}
  \setlength{\parsep}{0pt}
}{\end{enumerate}}

\newenvironment{packed_itemize}{
\begin{itemize}
  \setlength{\itemsep}{1pt}
  \setlength{\parskip}{0pt}
  \setlength{\parsep}{0pt}
}{\end{itemize}}

\makeatletter
\newcommand{\leqnomode}{\tagsleft@true}
\newcommand{\reqnomode}{\tagsleft@false}
\makeatother

\newtheorem{lemma}{Lemma} [section]
\newtheorem{theorem}{} [section]
\newtheorem{question}{Question}[section]
\newcommand{\sset}[1]{\left\{#1\right\}}

\newcommand{\sat}[1]{\textsc{$#1$-Sat}}
\newcommand{\col}[1]{\textsc{$#1$-Coloring}}

\def\longbox#1{\parbox{0.85\textwidth}{#1}}

\title{List-three-coloring graphs with no induced $P_6+rP_3$}
\author{
  Maria Chudnovsky\thanks{Supported by NSF grant DMS-1550991.
This material is based upon work supported in part by the U. S. Army  Research 
Laboratory and the U. S. Army Research Office under    grant number 
W911NF-16-1-0404.}\\
Princeton University, Princeton, NJ 08544
\\
\\
Shenwei Huang\\
Wilfrid Laurier University, Waterloo, ON, Canada
\\
\\
Sophie Spirkl\\
Princeton University, Princeton, NJ 08544
\\
\\
Mingxian Zhong\\
Columbia University, New York, NY 10027}

\date{\today}
\begin{document}
\maketitle

\begin{abstract} 
  For an integer $r$, the graph $P_6+rP_3$ has $r+1$ components, one of which is a path on $6$ vertices, and each of the others is a path on $3$ vertices. In
  this paper we provide a polynomial-time algorithm to test if, for fixed $r$,
  a graph  with no induced subgraph isomorphic to $P_6+rP_3$ is
  three-colorable, and find a coloring if one exists. We also
  solve the list version of this problem, where each vertex is assigned a list
  of possible colors, which is a subset of $\{1,2,3\}$.
\end{abstract}

\section{Introduction} \label{sec:intro}

All graphs in this paper are finite and simple.
We use $[k]$ to denote the set $\sset{1, \dots, k}$. Let $G$ be a graph. A
\emph{$k$-coloring} of $G$ is a function
$f:V(G) \rightarrow [k]$ such that for every edge 
$uv \in E(G)$,
$f(u) \neq f(v)$, and $G$ is \emph{$k$-colorable} if $G$ has a
 $k$-coloring. The \textsc{$k$-coloring problem} is the problem of
deciding, given a graph $G$, if $G$ is $k$-colorable. This
problem is well-known to be $NP$-hard for all $k \geq 3$.

A function $L: V(G) \rightarrow 2^{[k]}$ that assigns a subset of
$[k]$ to each vertex of a graph $G$ is a \emph{$k$-list assignment}
for $G$. For a $k$-list assignment $L$, a function
$f: V(G) \rightarrow [k]$ is a 
\emph{coloring of $(G,L)$} if $f$ is a
$k$-coloring of $G$ and $f(v) \in L(v)$ for all $v \in V(G)$. 
We say that a graph $G$ is  \emph{$L$-colorable},
and that the pair $(G,L)$ is {\em colorable},  if $(G,L)$ has a coloring.
The \textsc{$k$-list coloring problem} is the problem of deciding, given a
graph $G$ and a $k$-list assignment $L$, if $(G,L)$ is 
colorable. Since this generalizes the $k$-coloring problem, it is
also $NP$-hard for all $k \geq 3$.

Let $G$ be a graph, and let $X \subseteq V(G)$. We denote by $G|X$ the subgraph
of $G$ induced by $X$. For a list assignment $L$ for $G$,
a \emph{precoloring} $(G, L, X, f)$ of $(G,L)$ is
a function $f: X \rightarrow \mathbb{N}$ for a set $X \subseteq V(G)$ such
that $f(v) \in L(v)$ for every $v \in X$, and $f$
is a $k$-coloring of $G|X$.
A \emph{precoloring extension} for $(G, L, X, f)$ is a
$k$-coloring $g$ of $(G,L)$ such that $g|_X = f|_X$.
The \textsc{precoloring extension problem} is the problem of deciding
if a given precoloring $(G,L,X,f)$ of $(G,L)$ extends to a coloring of
$(G,L)$.

We denote by $P_t$ the path with $t$ vertices.  
Given a path $P$, its \emph{interior} is the set of vertices that
have degree two in $P$. 
A  \emph{$P_t$ in a graph  $G$} is  a sequence $v_1-\ldots -v_t$ of pairwise 
distinct vertices where for $i,j \in [t]$, $v_i$ is adjacent to $v_j$ if and 
only if $|i-j|=1$.  We denote by $V(P)$ the set $\{v_1, \ldots, v_t\}$,
and if $a, b \in V(P)$, say $a=v_i$ and $b=v_j$ and $i<j$, then $a-P-b$ is the 
path $v_i-v_{i+1}-\ldots - v_j$. We denote by $P_6+rP_3$ the graph
with $r+1$ components, one of which is a $P_6$, and each of the others is a
$P_3$.

For two graphs $H,G$ we say that $G$ is {\em $H$-free} if no induced
subgraph of $G$ is isomorphic to $H$. In this paper, we use the terms
``polynomial time'' and ``polynomial size'' to mean ``polynomial in
$|V(G)|$'', where $G$ is the input graph.
Since the  \textsc{$k$-coloring problem} and the \textsc{$k$-precoloring extension problem} are $NP$-hard for $k \geq 3$, their
restrictions to $H$-free graphs, for various $H$,  have been extensively 
studied. In particular, the following is known:

\begin{theorem}[\cite{gps}]
\label{linearforests}
  Let $H$ be a (fixed) graph, and let $k>2$. If
the \textsc{$k$-coloring problem} can be solved in polynomial time when restricted to the class of $H$-free graphs, then every connected component of $H$ is a path.
\end{theorem}

In this paper we focus on the case when $k=3$.  In this case, the converse of
\ref{linearforests} may be true (it is known to be false for $k>4$), since the
following question is still open:

\begin{question} Is it true that for every (fixed) integer $t>0$,
  the \textsc{ $3$-coloring problem} can be solved
  in polynomial time when restricted to the class of $P_t$-free graphs?
\end{question}

Several positive results are known in this direction:

\begin{theorem}[\cite{c1}] \label{3colP7}
The \textsc{list-3-coloring problem} can be
  solved in polynomial time for the class of $P_7$-free graphs.
\end{theorem}

\begin{theorem}[\cite{PrPspaper}] \label{PrPs}
The \textsc{list-3-coloring problem} can be
solved in polynomial time for the class of $P_5+P_2$-free graphs, and for
the class of $P_4+P_3$-free graphs.
\end{theorem}

\begin{theorem}[\cite{kP3paper}] \label{kP3}
Let $H$ be a graph each of whose components is a $P_3$. 
  The \textsc{3-coloring problem} can be
  solved in polynomial time for the class of $H$-free graphs.
  \end{theorem}

\begin{theorem} [\cite{subexp}]
  The \textsc{list-3-coloring problem} can be
  solved in subexponential time $2^{O(\sqrt{t|V(G)|\log(|V(G)|)})}$ where
  the input graph $G$ is $P_t$-free.
\end{theorem}

Our main result is the following:
\begin{theorem} 
\label{main}
The \textsc{list-3-coloring problem} can be
  solved in polynomial time for the class of $(P_6+rP_3)$-free graphs.
\end{theorem}

This immediately implies that:
\begin{theorem} 
The \textsc{3-coloring problem} can be
  solved in polynomial time for the class of $(P_6+rP_3)$-free graphs.
\end{theorem}

In contrast, we also show that

\begin{theorem}
\label{P5P2}
The \textsc{$k$-coloring problem} restricted to
$P_5+P_2$-free graphs is $NP$-hard for $k \geq 5$.
\end{theorem}

This paper is organized as follows. Section~\ref{sec:tools} is a collection of tools that we use in the proof. Section~\ref{sec:seeds} describes the main
object we work with, a ``$r$-seeded precoloring''.
A $r$-seeded precoloring consists of a graph $G$, a precolored subset $S$
of vertices, and a list of allowed colors for every vertex of $V(G)$.
Sections~\ref{sec:generaltonice},
\ref{sec:nicetostable} and \ref{sec:stableto2SAT} contain a sequence of
theorems that
start with a general $r$-seeded precoloring, and, by ``guessing'' 
(by exhaustive enumeration) the coloring of
a certain bounded-size set of vertices, transform it into
a precoloring that is ``tractable''. Here by tractable we mean
a precoloring for which the precoloring extension problem can be solved in
polynomial time.
Section~\ref{sec:complete} combines the results of the previous three
sections to obtain a proof of
\ref{main}. Finally Section~\ref{sec:hardness} is devoted to the proof of
\ref{P5P2}.

\section{Tools}\label{sec:tools}

In this section we discuss several tools that we repeatedly  use in this paper.
The first is a result of \cite{edwards}:
\begin{theorem}[\cite{edwards}]
  \label{Edwards}
Let $G$ be  a graph, and let $L$ be a list assignment for $G$ 
such that $|L(v)|\leq 2$
for all $v\in V(G)$. Then a coloring of $(G,L)$, or a
determination that none exists, can be obtained in time
$O(|V(G)|+|E(G)|)$.
\end{theorem}

We also need a modification of \ref{Edwards}.
For a graph  $G$, a coloring $c$  of $G$,  and a set $X \subseteq V(G)$,
we say that $X$ is    \emph{monochromatic in $c$}
if $c(u)=c(v)$ for all $u,v\in X$. Let $L$ be a list assignment for  $G$, and
$\mathcal{X}$ a set of subsets of $V(G)$. We say that the triple
$(G,L,\mathcal{X})$  is
\emph{colorable} if there is a coloring $c$ of $(G,L)$ such that
$X$ is monochromatic in $c$ for all $X \in \mathcal{X}$.
We need the following.

\begin{theorem} [\cite{Stacho}]
  \label{Mono}
Let $G$ be a graph, and let $L$ be a list assignment for $G$ 
such that $|L(v)|\leq 2$
for all $v\in V(G)$. Let $\mathcal{X}$ be  a set  of subsets $V(G)$ where
$|\mathcal{X}|$ is polynomial. Then a
coloring of $(G,L,\mathcal{X})$, or a determination that none exists, can be
obtained in polynomial time.
\end{theorem}

Note that if two sets $X$ and $X'$ with $X \cap X' \neq \emptyset$ are
monochromatic in a coloring $c$,  then $X \cup X'$ is also monochromatic in $c$.
Thus, given a triple  $(G,L,\mathcal{X})$ as in \ref{Mono}
we can compute in polynomial time a triple 
$(G,L,\mathcal{X}')$ where the sets in $\mathcal{X}'$ are pairwise disjoint and
$(G,L, \mathcal{X})$ has a coloring  if and
only if $(G, L, \mathcal{X}')$ does. Thus \ref{Mono} follows from

\begin{theorem} [\cite{Stacho}]
  \label{Monodisjoint}
Let $G$ be a graph, and let $L$ be a list assignment for $G$ 
such that $|L(v)|\leq 2$
for all $v\in V(G)$. Let $\mathcal{X}$ be  a set of pairwise
disjoint subsets of $V(G)$. Then a
coloring of $(G,L,\mathcal{X})$, or a determination that none exists, can be
obtained in time $O(|V(G)|+|E(G)|)$. 
\end{theorem}


Next we present  a result from \cite{DSW}. A  {\em hypergraph} $H$ consists of a finite set $V(H)$ of vertices and a set $E(H)$ of non-empty subsets of
$V(H)$ called
{\em hyperedges}.  A {\em matching} in $H$ is a set of
pairwise disjoint non-empty hyperedges, and a {\em hitting set} in $H$ is a
set of vertices meeting every hyperedge. We denote by $\nu(H)$ the maximum size
of a matching in $H$, and by $\tau(H)$ the minimum size of a hitting set in $H$.
The parameters $\nu(H)$ and $\tau(H)$ are well-known, but we need one more.
We denote by $\lambda(H)$ the maximum $k\geq 2$ such that there are edges
$e_1, \ldots, e_k \in E(H)$ with the property that for every $i,j$ with 
$1 \leq i <j \leq k$ there exist $v_{i,j} \in V(H)$ satisfying
$\{h: 1 \leq h \leq k  \text{ such that } v_{i,j} \in e_h\}=\{i,j\}.$
If there is no such $k$, we set $\lambda(H)=2$. We need the following:

\begin{theorem}[\cite{DSW}]
\label{nutaulambda}
For every hypergraph $H$, 
$\tau(H) \leq 11\lambda(H)^2(\lambda(H)+\nu(H)+3){{\lambda(H)+\nu(H)} \choose {\nu(H)}}^2.$
\end{theorem}

Let us discuss how we apply \ref{nutaulambda}. Let $G$ be a graph. 
A set $X \subseteq V(G)$ is {\em stable} if no edge of $G$ has both its ends 
in $X$. Let  $A \subseteq V(G)$.
An \emph{attachment} of $A$ is a vertex of $V(G) \setminus A$ with a neighbor in
$A$.
For $B \subseteq V(G) \setminus A$ we denote by $B(A)$ the set of
attachments of $A$ in $B$.  If $F=G|A$, we sometimes
write $B(F)$ to mean $B(V(F))$.

\begin{theorem}
  \label{smalllambda}
Let $t$ be an integer and let $G$ be a $P_t$-free graph, and let
$X,Y \subseteq V(G)$ be disjoint, where $X$ is stable and
every component of $Y$ has size at most $p$. Let $\mathcal{Z}$  be
a set of connected subsets of size $q$ of $Y$, each of which has an
attachment in $X$.
Let $H$ be a hypergraph with vertex set $X$ and hyperedge set
$\{ X(Z) \text { : } Z \in \mathcal{Z}\}$. Then
$\lambda(H) \leq  {p \choose q} \lceil{{t+1} \over 2}\rceil.$
\end{theorem}

\begin{proof}
  Let $\lambda=\lambda(H)$ and let $e_1, \ldots, e_{\lambda}$, and
  $\{v_{i,j}\}_{1 \leq i<j\leq \lambda}$ be
  as in the definition of $\lambda(H)$. For $i \in \{1, \ldots, \lambda\}$,
  let $Y_i \subseteq \mathcal{Z}$ be such that $e_i=X(Y_i)$.
  Define a graph $F$ with vertex set  $\{e_1, \ldots, e_{\lambda}\}$ and such
  that  $e_i$ is adjacent to $e_j$ in $F$ if
either $Y_i \cap Y_j \neq \emptyset$, or 
  in $G$ there is an edge
  with one end in $Y_i$ and the other end in $Y_j$. Then
  $deg_F(e_i) \leq {p \choose q} -1$ for every $e_i \in V(F)$. It follows that
  $F$ is $p \choose q$-colorable, and so $F$ has a stable set $S$ with
  $|S| \geq {\lambda \over {p \choose q}} \geq \lceil {{t+1} \over 2} \rceil.$ Write
  $m=\lceil{{t+1} \over 2}\rceil$.
   Renumbering if necessary, we
  may assume that $e_1, \ldots, e_m \in S$.
  Let $q_1$ be a neighbor of $v_{1,2}$ in $Y_1$, and let $q_m$
  be a neighbor of $v_{m-1,m}$ in $Y_m$. For $i \in \{2, \ldots, m-2\}$
  let $Q_i$ be a path from $v_{i-1,i}$ to $v_{i,i+1}$ with interior
  in $Y_i$ (such a path exists by the definition of $H$). Now
  $q_1-v_{1,2}-Q_2-v_{2,3}-\ldots-v_{m-2,m-1}-Q_{m-1}-v_{m-1,m}-q_m$ is a path of length at least $t$ in $G$, a contradiction. This proves~\ref{smalllambda}.
\end{proof}

We deduce
\begin{theorem}
  \label{nutau}
  Let $r$ be an integer, and let $G$ be $(P_6+rP_3)$-free. Let $H$ be
  a hypergraph as in \ref{smalllambda}. Let $C={p \choose q} (2r+4)$.
  Then 
  $\tau(H) \leq 11C^2(C+\nu(H)+3){{C+\nu(H)} \choose {\nu(H)}}^2.$
  In particular, there is a function
  $f_{r,p,q}: \mathbb{N} \rightarrow \mathbb{N}$
  such that $\tau(H) \leq f_{r,p,q}(\nu(H))$.
\end{theorem}

\begin{proof}
  Since $P_6+rP_3$ is contained in $P_{4r+6}$, and since
  $G$ is $(P_6+rP_3)$-free, it follows that $G$ is $P_{4r+6}$-free.
  By \ref{smalllambda}, $\lambda(H) \leq C$. But now \ref{nutau}
  follows directly from \ref{nutaulambda}.
  \end{proof}

We finish this section with some terminology.
Let $G$ be a graph. For $X \subseteq V(G)$ we denote by
$G \setminus X$ the graph $G|(V(G) \setminus X)$.
If $X=\{x\}$, we write $G \setminus x$ to mean $G \setminus \{x\}$.
For disjoint subsets $A,B \subset V(G)$ we say that $A$ is \emph{complete} to $B$ if every vertex of $A$ is adjacent to every vertex of $B$, and that 
$A$ is \emph{anticomplete} to $B$ if every vertex of $A$ is non-adjacent to
every vertex of $B$. If $A=\{a\}$ we write $a$ is complete (or anticomplete)
to $B$ to mean $\{a\}$ that is complete (or anticomplete) to $B$.
If $a \not \in B$ is not complete and not anticomplete to $B$,
we say that $a$ is \emph{mixed} on $B$. Finally, if $H$ is an induced subgraph
of $G$ and $a \in V(G) \setminus V(H)$, we say that $a$ is \emph{complete to, 
anticomplete to}, or \emph{mixed on} $H$ if $a$ is complete to, anticomplete 
to,  or mixed on 
$V(H)$, respectively. For $v \in V(G)$ we write $N_G(v)$ (or $N(v)$ when there is no danger of confusion) to mean the set of vertices of $G$ that are adjacent to  $v$. Observe that since $G$ is simple, $v \not \in N(v)$.  
For $X \subseteq V(G)$ a {\em component}
of $X$ (or of $G|X$) is the vertex set of a maximal connected subgraph of
$G|X$.

Let $L$ be a list assignment for $G$. We denote by $X^0(L)$ the set
of all vertices $v$ with $|L(v)|=1$. For $X \subseteq V(G)$, we write
$(G|X,L)$ to mean the list coloring problem where we restrict the
domain of the list assignment $L$ to $X$.  Let $X \subset X^0(L)$,
and let $Y \subset V(G)$. We say that a list assignment $M$
is \emph{obtained from $L$ by updating
 $Y$ from $X$} if $M(v)=L(v)$ for every $v \not \in Y$, and
$M(v)=L(v) \setminus \bigcup_{x \in N(v) \cap X} L(x)$ for every
$v \in Y$.  If $Y=V(G)$, we say that $M$ is \emph{obtained from $L$ by
  updating from $X$}.  If $M$ is obtained from $L$ by updating from
$X^0(L)$, we say that $M$ is \emph{obtained from $L$ by updating}.
For $v \in X^0(L)$ we will  not distinguish between the set $L(v)$ and its
unique element. For $X \subseteq X^0(L)$, we will regard $L$ as a coloring of
$G|X$.
Let
$L_0=L$, and for $i \geq 1$ let $L_i$ be obtained from $L_{i-1}$ by
updating. If $L_i=L_{i-1}$, we say that $L_i$ is \emph{obtained from
  $L$ by updating exhaustively}.
Since
$0 \leq \sum_{v \in V(G)} |L_j(v)| < \sum_{v \in V(G)} |L_{j-1}(v)| \leq 3|V(G)|$
for all $j < i$,
it follows that $i \leq 3 |V(G)|$ and
thus $L_i$ can be computed from $L$ in polynomial time.
This observation allows us to set the following convention.

\begin{theorem}
  \label{convention}
If $G$ is a graph, $L$ a list assignment for $G$, and $v \in V(G)$, then 
there is no $u \in N(v) \cap X^0(L)$ with $L(u) \subseteq L(v)$.
\end{theorem}

A {\em seagull} $S$ in $G$ is a $P_3$ $a-b-c$ in $G$. We write $V(S)=\{a,b,c\}$.
    The vertices $a$ and $c$
    are called the {\em wings} of the seagull, and $b$ is the {\em body} of the seagull. For $X,Y \subseteq V(G)$, $S$ is an $X$-seagull if $V(S) \in X$,
    and $S$ is an $(X,Y)$-seagull if $S$ has one wing in $X$, and the body and the other wing in $Y$.
      A {\em flock} is a set of pairwise disjoint seagulls that are pairwise anticomplete to each other. The {\em size} of a flock is its cardinality.

\section{Seeded precolorings} \label{sec:seeds}

Given a $(P_6+rP_3)$-free graph $G$ with a $3$-list assignment $L$, our strategy
for checking if $(G,L)$ is colorable   involves several steps, each of which
consists of choosing a small subset
$S \subseteq V(G)$, precoloring  $S$ and updating the lists. In view of
\ref{Edwards}, if we arrive at a situation where every vertex has list
of size at most two, then we are done. We show, roughly, that this can always
be achieved. To keep track of the precoloring and updating process, we
define the following object.

A \emph{$r$-seeded precoloring} of the pair $(G,L')$ is a triple 
$P = (G, L,S)$ such that 
\begin{packed_enum}
\item $L$ is a list assignment for $G$, and $L(v) \subseteq \{1,2,3\}$ for
  every $v \in V(G)$, and $|L(v)|=1$ for every $v \in S$,
\item $(G, L', S, L)$ is a precoloring of $(G,L')$;
\item   $G|S$ contains $P_6+(r-1)P_3$.
\end{packed_enum}
We call $S$ the {\em seed} of $P$, and write $S(P)$ to mean $S$.
Similarly, we use the notating $G(P)$  and $L(P)$.
The {\em boundary} $B(P)$ of $P$ is the set of all vertices
$v \in V(G)$  with $|L(v)|=2$ and such that
$v$ has a neighbor $s \in S$ with $L(s)=\{1,2,3\} \setminus L(v)$.
We denote by $B(P,i)$ the set of vertices $b \in B$ with $i \in L(b)$,
and by $B(P)_{ij}$ the set of vertices $b \in B(P)$ with $L(b)=\{i,j\}$.
Finally,  the {\em wilderness} $W(P)$ of $P$ is the set
$V(G) \setminus (X^0(L) \cup B(P))$.
The reason for the name ``wilderness'' is that every $v \in V(G)$ with
$|L(v)|=3$
belongs to $W(P)$, and so by \ref{Edwards} $W(P)$ is the set where the
algorithmic difficulty lies.

We observe the following.
\begin{theorem}
  \label{W}
  Let $r$ be an integer,
  let $G$ be a $P_6+rP_3$-graph and let $P=(G,L,S)$ be an $r$-seeded precoloring of $G$. Assume that $G$ does not contain a clique of size four.
  Then each component of $W(P)$ is a clique of size at most three.
\end{theorem}

\begin{proof}
  Since $G|S$ contains $P_6+(r-1)P_3$, it follows that $G|W$ is $P_3$-free.
  Consequently every component of $W$ is a clique, and since $G$ has no clique of size four, \ref{W} follows.
  \end{proof}

Let $P=(G,L,S)$ be an $r$-seeded precoloring. Let  $U \subseteq V(G)$ and a
let $c$ be a coloring of $G|U$. We say that the seeded precoloring
$P'=(G,L',S')$ is
obtained from $P$ by {\em moving $U$ to the seed with  $c$} if
$S'=S \cup U$, and $L'$ is obtained by updating exhaustively from the list
assignment $L''$, defined as follows:
$L''(u)=c(u)$ for every $u \in U$, and $L''(v)=L(v)$
for every $v \in V(G) \setminus U$.

For an $r$-seeded  precoloring $P$ and a collection
 $\mathcal{L}$ of $r$-seeded precolorings, we say that $\mathcal{L}$ is an
\emph{equivalent collection} for $P$ (or that $P$ is \emph{equivalent}
to $\mathcal{L}$) if $P$ has a precoloring extension if and only if at
least one of the precolorings in $\mathcal{L}$ has a precoloring
extension, and a precoloring extension of $P$ can be constructed from
a precoloring extension of a member of $\mathcal{L}$ in polynomial
time.

Let $P=(G,L,S)$ be an $r$-seeded precoloring. A {\em type} is a non-empty
monochromatic
subset of $S$. Thus for every $b \in B(P)$, $N(b) \cap S$ is a type; we
call $N(b) \cap S$ the {\em type of $b$}. For $S' \subseteq S$ we
denote by $B(P,S')$ the set of all vertices of $B(P)$ whose type includes $S'$.
In what follows we will often need to  handle  each type of
$S$ separately, and so it is important that we keep track of the size of
the seed in every precoloring we consider.

\section{Nice and easy  precolorings}\label{sec:generaltonice}

A $r$-seeded precoloring $P$ is {\em nice} if 
no vertex of $B(P)$ is mixed on an edge of $W(P)$, and it is
{\em easy} if $G|(B(P) \cup W(P))$ is $P_6$-free.
    Our first goal is to show that an $r$-seeded precoloring
    $P$ can be replaced by an equivalent collection of precolorings
    each of which is either nice or easy, such that the size of the
    collection is polynomial, and the size of the seed of each of its members is bounded by a function of $|S(P)|$.
    For a precoloring extension $c$ of $P$, we will define several
    ``characteristics'' of $c$. While we cannot
enumerate all precoloring extensions of $P$ (in polynomial time), it is
possible to enumerate all characteristics, and that turns out to be enough
for our purposes.

Thus let $P=(G,L,S)$ be an $r$-seeded precoloring and let $c$ be a precoloring
extension of $P$. 
For every $i \in \{1,2,3\}$
we define the hypergraph $H(P,i,c)$ as follows. The vertex set
 $V(H(P,i,c))=\{b \in B(P,i) \text{ : } c(b)=i\}$. Next we construct
the hyperedges. Let $K$  be the set of all  edges  $w_1w_2$ of $G$ with both ends
in $W(P)$ such that some vertex of $V(H(P,i,c))$ is mixed on
$\{w_1,w_2\}$. For every $e=w_1w_2 \in K$,  let $h(e)$ be the set of
attachments of $\{w_1,w_2\}$ in $V(H(P,i,c))$. Then
$\{h(e) \text{ : } e \in K  \}$ is the set of the hyperedges of
  $H(P,i,c)$.

\begin{theorem}
  \label{neatlemma}
  There is a function $f:\mathbb{N} \rightarrow \mathbb{N}$ with the following
  properties. 
  Let $r>0$ be an integer, $G$ a $(P_6+rP_3)$-free graph with no clique of size four,  $P=(G,L,S)$
  an $r$-seeded precoloring, and let $c$ be a coloring of $(G,L)$.
Write $M=2^{|S|}(r+6)$.
  Then for
  every $i \in \{1,2,3\}$ either
  \begin{packed_enum}
  \item there exists $X \subseteq V(H(P,i,c))$ with $|X| \leq f(M)$ where 
    for every edge $w_1w_2$ of $G|W(P)$ such that some vertex of
    $V(H(P,i,c))$ is mixed on $\{w_1,w_2\}$, at least  one of $w_1,w_2$ has a neighbor in $X$,
     or
  \item there exists a flock $F=\{a_1-b_1-c_1, \ldots, a_{M}-b_{M}-c_{M}\}$
    where for every $i$, $a_i \in V(H(P,i,c))$ and
    $b_i,c_i \in W(P)$, and
    such that every vertex of $V(H(P,i,c))$  has a neighbor in
    $\{b_i,c_i\}$
    for at most one value of $i$.
  \end{packed_enum}
\end{theorem}

\begin{proof}
By \ref{W} every component of $W$ is a clique of size at most three.
  Let $f=f_{r,3,2}$  be as in \ref{nutau}.
  Applying \ref{nutau} to $H=H(P,i,c)$ with $p=3$ and $q=2$, we deduce that
  either $\nu(H) \geq M$ or $\tau(H) \leq f(M)$.

  Suppose first  that $\nu(H) \geq M$. Let $b_1c_1, \ldots, b_Mc_M$ be
  edges of $G|W$
  such that $M=\{h(b_1c_1), \ldots, h(b_{M},c_{M})\}$ is a
  matching of $H$.
  It follows from the definition of $h(b_jc_j)$ (using symmetry) that
  for every $j$ there exists $a_j \in V(H(P,i,c))$ such that $a_j-b_j-c_j$ is a
  seagull. Moreover, since $M$ is a matching of $H$, no $v \in V(H(P,i,c)$
  belongs to more than one  $h(b_jc_j)$, and therefore every vertex of
  $V(H(P,i,c))$  has a neighbor in   $\{b_j,c_j\}$
  for at most one value of $j$. This proves that  if $\nu(H) \geq M$, then
  \ref{neatlemma}.2 holds.

  Thus we may assume that $\tau(H) \leq f(M)$. Letting
  $X \subseteq V(H(P,i,c))$ be a hitting set for $H$, we immediately see that
  \ref{neatlemma}.1 holds.
\end{proof}

Given an $r$-seeded precoloring $P$,
$i \in \{1,2,3\}$ and  a precoloring
extension $c$ of $P$,
we say that an {\em $M$-characteristic}  of $P,i,c$ (denoted by
$char_M(P,i,c)$) is $X$ if
\ref{neatlemma}.1 holds for $P,i$  and $c$,  and $F$ if
\ref{neatlemma}.2
holds for $P,i$ and $c$.  We denote by $V(char_M(P,i,c))$ the
set of all the vertices involved in $char_M(P,i,c)$.

We also need a version of the hypergraph above for each type, as follows.
For every type $T \subseteq S$ and every $i \in \{1,2,3\}$
we define the hypergraph $H(P,T,i,c)$. The vertex set
 $V(H(P,T,i,c))=\{b \in B(P,T) \text{ : } c(b)=i\}$. Next we construct
the hyperedges. Let $K$  be set of all  edges  $w_1w_2$ of $G$ with both ends
in $W(P)$ such that some vertex of $V(H(P,T,i,c))$ is mixed on
$\{w_1,w_2\}$. For every $e=w_1w_2 \in K$,  let $h(e)$ be the set of
attachments of $\{w_1,w_2\}$ in $V(H(P,T,i,c))$. Then
$\{h(e) \text{ : } e \in K  \}$ is the set of the hyperedges of
  $H(P,T,i,c)$.

\begin{theorem}
  \label{neatlemmatype}
  There is a function $f:\mathbb{N} \rightarrow \mathbb{N}$ with the following
  properties.
  Let $r>0$ be an integer, $G$ a $(P_6+rP_3)$-free graph with no clique of size four,  $P=(G,L,S)$
  an $r$-seeded precoloring and let $c$ be a precoloring of $G$. Then for
  every type $T$ of $S$ either
  \begin{packed_enum}
  \item there exists $X \subseteq V(H(P,T,i,c))$ with $|X| \leq f(2)$ where
    for every edge $w_1w_2$ of $G|W(P)$ such that some vertex of
    $V(H(P,T,i,c))$ is mixed on $\{w_1,w_2\}$, at least  one of $w_1,w_2$ has a neighbor in $X$,
  \item there exists a flock $F=\{a_1-b_1-c_1, a_2-b_2-c_2\}$
    with $a_1,a_2 \in V(H(P,T,i,c))$, and
    $b_1,b_2,c_1,c_2 \in W(P)$.
  \end{packed_enum}
\end{theorem}

\begin{proof}
By \ref{W} every component of $W$ is a clique of size at most three.  Let $f=f_{r,3,2}$  be as in \ref{nutau}.
  Applying \ref{nutau} to $H=H(P,T,i,c)$ with $p=3$ and $q=2$, we deduce that
  either $\nu(H) \geq 2$ or $\tau(H) \leq f(2)$.
  Suppose first  that $\nu(H) \geq 2$. Let $b_1c_1, b_2,c_2$ be
  edges of $G|W$
  such that $M=\{h(b_1c_1), h(b_2c_2)\}$ is a matching of $H$.
  It follows from the definition of $h(b_ic_i)$ (using symmetry)  that
  for every $i$ there exists $a_i \in V(H(P,T,i,c))$ such that
  $a_i-b_i-c_i$ is a
  seagull, and
  \ref{neatlemmatype}.2 holds.
  Thus we may assume that $\tau(H) \leq f(2)$. Letting
  $X \subseteq V(H(P,T,i,c))$ be a hitting set for $H$, we immediately see that
  \ref{neatlemmatype}.1 holds.
\end{proof}

Given an $r$-seeded precoloring $P$, a type $T$ of $S(P)$,
$i \in \{1,2,3\}$ and  a precoloring
extension $c$ of $P$,
we say that an {\em $2$-characteristic}  of $P,T,i,c$ (denoted by
$char_2(P,T,i,c)$) is $X$ if
\ref{neatlemmatype}.1 holds for $P,T,i$  and $c$,  and $F$ if \ref{neatlemmatype}.2
holds for $P,T,i$ and $c$.  We denote by $V(char_2(P,T,i,c))$ the
set of all the vertices involved in $char_2(P,T,i,c)$.

We can now prove the main result of the section.

\begin{theorem}
  \label{generaltonice}
  There exists a  function $g_1: \mathbb{N} \rightarrow \mathbb{N}$
   with the following properties.
  Let $G$ be a $(P_6+rP_3)$-free graph with no clique of size four, and let $P=(G,L,S)$ be  $r$-seeded
  precoloring of $G$. There is a collection $\mathcal{L}$
  of  $r$-seeded precolorings such that
  \begin{packed_enum}
  \item $\mathcal{L}$ is equivalent to $P$
\item every $P' \in \mathcal{L}$ is either nice or easy
  \item $|S(P')| \leq g_1(|S(P)|)$ for every $P' \in \mathcal{L}$
  \item $|\mathcal{L}| \leq |V(G)|^{g_1(|S|)}$
  \end{packed_enum}
  Moreover, given $P$, the collection $\mathcal{L}$ can be constructed in
  time $O(|V(G)|^{g_1(|S|)})$.
\end{theorem}

\begin{proof}
  Let $f$ be as in \ref{neatlemma} and let  $M=2^{|S|}(r+6)$.
For an $r$-seeded precoloring $P'$ of $(G,L)$ and for $i \in \{1,2,3\}$
  let $smallguess(P',i)$ be the set of all subsets of
  $B(P',i)$ of size at most $f(M)$, and $bigguess(P',i)$ be the set
  of all flocks of size $M$ such that every seagull of the flock
is $(B(P',i),W(P'))$-seagull.
Let $guess(P',i)=smallguess(P',i) \cup bigguess(P',i)$.
Thus $guess(P',i)$ is the set of all possible $M$-characteristics of a
precoloring extension of $P'$.
We say that $X_i \in guess(P',i)$ is {\em small} if
  $X_i \in smallguess(P',i)$ and that $X_i$ is {\em  big} if
  $X_i \in bigguess(P',i)$. If $X_i$ is big, we denote
  by $U_{i}$ the set of the wings of the flock that are contained in
  $B(P',i)$, by $W_{i}$ the vertices of the set of the bodies and the
  wings of the flock that are contained in $W(P')$, and write
  $V(X_i)=V_{i}=U_{i} \cup W_{i}$.  If $X_{i}$ is small, we write
  $V(X_i)=V_{i}=U_{i}=X_{i}$, and $W_{i}=\emptyset$. 
Thus in both cases $W_i=V(X_i) \cap W(P')$ and $U_i=V(X_i) \cap B(P')$.
  A precoloring $c$ of $(G|V(X_i), L)$ is {\em $i$-consistent} if
  $c(v)=i$ for every $v \in U_{i}$.
  
Let $X_i \in guess(P',i)$, and let
 $c$ be  an $i$-consistent precoloring of
  $(G|V_i,L(P'))$.
 We define the $r$-seeded precoloring $P'(X_i,c)$. Let
 $\tilde{P}=(G,L',S')$ be obtained from $P'$ by moving $X_i$ to the seed with
 $c$. Next we modify $L'$ further.
   \begin{packed_itemize}
 \item Assume first that $X_{i}$ is small.
   If $b \in B(\tilde{P},i) \cap B(P')$ and
 $b$ is  mixed on an edge
   of $W(\tilde{P})$, remove $i$ from  $L'(b)$.
 \item Next assume that $X_{i}$ is big. Let
   $X_{i}=\{a_1-b_1-c_1, \ldots, a_{M}-b_{M}-c_{M}\}$ where
   $U_i=\{a_1, \ldots, a_M\}$.
   If $b \in B(\tilde{P},i) \cap B(P')$ and $b$ has
   a neighbor in $\{b_q,c_q\}$ for more than one value of $q$, remove $i$
   from $L'(b)$.
 \end{packed_itemize}
   Let $P'(X_i,c)$ be the $r$-seeded precoloring thus obtained. Note that
   given $P'$, the $r$-seeded precoloring $P'(X_i,c)$ can be constructed in
   polynomial time.

 We now proceed as follows. For every $X_1 \in guess(P,1)$ and
 every $1$-consistent  coloring $c_1$ of $(G|V(X_1),L(P))$
 construct $P(X_1,c_1)$ as above. Now for every $X_2 \in guess(P(X_1,c_1),2)$
 and every $2$-consistent  coloring of $(G|V(X_2),L(P(X_1,c_1)))$ construct
 $P(X_1,c_1)(X_2,c_2)$ as above. Finally, for every $X_3 \in guess(P(X_1,c_1)(X_2,c_2),3)$ and every $3$-consistent coloring $c_3$ of $(G|V(X_3),L(P(X_1,c_1)(X_2,c_2)))$
 construct $P(X_1,c_1)(X_2,c_2)(X_3,c_3)$. Let
 $\mathcal{Q}$ be the set of triples $X=(X_1,X_2,X_3)$ where $X_1,X_2,X_3$ are
 as above.
 Write $V(X)=V_1 \cup V_2 \cup V_3$. A coloring $c$ of $V(X)$
 is {\em consistent} if $c_i=c|X_i$ is $i$-consistent.
 For every $X \in \mathcal{Q}$ and every consistent coloring
 of $G|V(X)$, let $Q_{X,c}=P(X_1,c_1)(X_2,c_2)(X_3,c_3)$.

We now list several properties of $Q_{X,c}$.

\begin{equation}
  \label{Q}
   \longbox{ 
     Let $i \in \{1,2,3\}$.
     \begin{packed_enum}
       \item $W(Q_{X,c})$ is anticomplete to $V(X)$.
   \item if $b \in B(Q_{X,c},i)$ then $b$ is anticomplete to $U_i$.
    \item if $b \in B(Q_{X,c},i)$ and $X_i$ is big, then $b$ has
      neighbors in at most one of the seagulls of $X_i$.
  \end{packed_enum}}
 \end{equation}
The first two statements of \eqref{Q} follow from the fact that
$V(X) \subseteq S(Q_{X,c})$. By the second bullet of the construction process of
$Q_{X,c}$, we deduce that
 if $X_i$ is big and  $b$ has neighbors in more than  one of the seagulls of
 $X_i$, then $i$ is removed from the list of $b$; thus the third statement of
 \eqref{Q} follows. This proves \eqref{Q}.

 \begin{equation}
   \label{acceptable}
   \longbox{ 
     Let $i,j \in \{1,2,3\}$  such that $i<j$. Then for every
     $X=(X_1,X_2,X_3) \in \mathcal{Q}$ where $X_j$ is big, the following hold.
  \begin{packed_itemize}
    \item $W_j$ is anticomplete to $V_i$, 
    \item if $u \in U_j$ and  $i \in L(P)(u)$, then $u$ is anticomplete to
      $U_i$, and
    \item if $u \in U_j$ and $i \in L(P)(u)$ and $X_i$ is big, then $u$ has
      neighbors in at most one of the seagulls of $X_i$.
  \end{packed_itemize}}
 \end{equation}
 Let $\tilde{P}$ denote the precoloring for which
 $X_j \in guess(\tilde{P},j)$. Then
 $V_i \in S(\tilde{P})$, and therefore $W_j$ is anticomplete to $V_i$,
 thus the first statement of \eqref{acceptable} holds.
 Next we prove the second and third statements.
 Let $u \in U_j$, then $u \in B(\tilde{P},j)$. Since $\tilde{P}$ was obtained
 from $P$ by moving vertices to the seed, it follows that
 $u \in B(P) \cup W(P)$. Recall that  by \ref{W} each component of
 $W(P)$ is a clique. Since $X_j$ is big, it follows that $u$ is mixed on an
 edge of $W(P)$, and therefore $u \in B(P)$. 
 Since $u \in B(\tilde{P},j) \cap B(P)$, it follows that
 $L(P)(u)=L(\tilde{P})(u)$.
 Assume that $i \in L(P)(u)$. Then $L(P)(u)=L(\tilde{P})(u)=\{i,j\}$.

 It follows that  $u$ has no neighbor in $S(\tilde{P})$ with color $i$.
 Since   $U_i \subseteq S(\tilde{P}) \subseteq S(Q_{X,c})$ and
 $L(\tilde{P})(U_i)=i$,
 the second statement of \eqref{acceptable} follows.
 By the second bullet of the construction process of $Q_{X,c}$, we deduce that
 if $X_i$ is big and  $u$ has neighbors in more than  one of the seagulls of
 $X_i$, then $i$ is removed from the list of $u$ during the process of constructing $\tilde{P}$; thus the third statement of
 \eqref{acceptable} follows. This proves \eqref{acceptable}.

\begin{equation}
   \label{smallguesslist}
\longbox{ Let $i \in \{1,2,3\}$.
     If $X_{i}$  is small, then no vertex  $b \in B(Q_{X,c},i)$
   is mixed on an edge of  $W(Q_{X,c})$.}
   \end{equation}
 Suppose $b \in B(Q_{X,c},i)$   
 is mixed on an edge of $W(Q_{X,c})$. Since $Q_{X,c}$ is obtained from
 $P$ by moving a set of vertices to the seed, it follows that
 $B(Q_{X,c}) \subseteq B(P) \cup W(P)$.
  Since by \ref{W} every component of $W(P)$ is a clique, it follows
 that no vertex of $B(Q_{X,c}) \setminus B(P)$ is mixed on an edge of $W(P)$,
 and therefore $b \in B(P)$. Consequently, $L(P)(b)=L(Q_{X,c})(b)$.
  However, in the construction process of $L(Q_{X,c})$, $i$ would be
 removed from the list of $b$, and thus $i \not \in L(Q_{X,c})(b)$,
   a contradiction. This proves~\eqref{smallguesslist}.

\bigskip

      Let $\mathcal{M}$ be the collection of all precolorings
      $Q_{X,c}$ as above. For every $Q_{Y,d} \in \mathcal{M}$ we proceed as
      follows. Write $Y=(Y_1,y_2,Y_3)$, $P'=Q_{Y,d}$ and $P'=(G,L',S')$.

  For every type $T$ of $S(P)$, and every $i \in \{1,2,3\} \setminus L(T)$,
  let $smallguess(T,i)$ be the set of all subsets of
  $B(P,T)$ of size at most $f(2)$ (here $f$ is as in \ref{neatlemmatype}), and
  $bigguess(T,i)$ be the set
  of all flocks of size $2$ such that every seagull of the flock
  is a $(B(P,T),W(P'))$-seagull. Please note that here we are referring to
  types of $P$, and not of $P'$.
  Let $guess(T,i)=smallguess(T,i) \cup bigguess(T,i)$.
  Let $\mathcal{T}$ be the set of all types of $S(P)$, say $|\mathcal{T}|=t$.
  Now let $\mathcal{C}$ be the set of all $2t$-tuples
  $X=(X_{T,i})$ where $T \in \mathcal{T}$, $i \in \{1,2,3\} \setminus L(T)$ and
  $X_{T,i} \in guess(T,i)$. We say that $X_{T,i}$ is {\em small} if
  $X_{T,i} \in smallguess(T,i)$ and that $X_{T,i}$ is {\em  big} if
  $X_{T,i} \in bigguess(T,i)$. If $X_{T,i}$ is big, we denote
  by $U_{T,i}$ the set of the wings of the flock that are contained in
  $B(P,T)$, by $W_{T,i}$ the vertices of the set of the bodies and the
  wings of the flock that are contained in $W(P')$, and write
  $V_{T,i}=U_{T,i} \cup W_{T,i}$.  If $X_{T,i}$ is small, we write
  $V_{T,i}=U_{T,i}=X_{T,i}$, and $W_{T,i}=\emptyset$.    Finally, let
  $V(X)=\bigcup_{T,i}V_{T,i}$.

  A precoloring $c$ of $(G|V(X), L')$ is {\em consistent} if
  $c(v)=i$ for every $v \in U_{T,i}$.
  Let $X \in \mathcal{C}$ and let $c$ be a consistent precoloring of
  $(G|V(X),L')$.
 We construct the $r$-seeded precoloring $P_{X,c}$. Let
 $P''=(G,L'',S'')$ be obtained from $P'$ by moving $V(X)$ to the seed with
 $c$. Next we modify $L''$ further.
 For every   $T \in \mathcal{T}$ and $i \in \{1,2,3\} \setminus L(T)$,
 proceed as follows.
  \begin{packed_itemize}
 \item Assume  that $X_{T,i}$ is small.
   If $b \in B(P'')$ and $T \subseteq N(b) \cap S$, and $b$ is  mixed on an
   edge of $W(P'')$, remove $i$ from  $L''(b)$.
 \end{packed_itemize}

 Let $P_{X,c}$ be the $r$-seeded precoloring thus obtained.
 Using \eqref{smallguesslist} we immediately deduce:
 
 \begin{equation}
   \label{bigsmall}
   \longbox{Let $T \in \mathcal{T}$ and let
     $\{i,j\} = \{1,2,3\} \setminus L(T)$.
     If  $X_{T,i}$ is big, then both $Y_i$ and $Y_j$ are big.}
 \end{equation}

 Next we prove a statement similar to \eqref{smallguesslist}.
 
 \begin{equation}
   \label{smallguess}
\longbox{ Let  $T \in \mathcal{T}$ and let $i \in \{1,2,3\} \setminus L(T)$.
     If $X_{T,i}$  is small, then no vertex  $b \in B(P_{X,c})$ such that 
   $T \subseteq N(b) \cap S(P_{X,c})$ is mixed on an edge of  $W(P_{X,c})$.}
   \end{equation}
 Suppose $b \in B(P_{X,c})$ with $T \subseteq N(b) \cap S(P_{X,c})$
 is mixed on an edge of $W(P_{X,c})$. Since $P_{X,c}$ is obtained from
 $P$ by moving a set of vertices to the seed, it follows that
 $B(P_{X,c}) \subseteq B(P) \cup W(P)$.
  Since by \ref{W} every component of $W(P)$ is a clique, it follows
 that no vertex of $B(P_{X,c}) \setminus B(P)$ is mixed on an edge of $W(P)$,
 and therefore $b \in B(P)$ and
 $L(P_{X,c})(b)=L(b)=\{1,2,3\} \setminus L(T)$. In
 particular $i \in L(P_{X,c}(b))$.
  However, in the construction process of $L(P_{X,c})$, $i$ would be
 removed from the list of $b$, 
   a contradiction. This proves~\eqref{smallguess}.

   \begin{equation}
     \label{niceoreasy}
   P_{X,c} \text{ is nice or easy.}
 \end{equation}
Write $B=B(P_{X,c})$ and $W=W(P_{X,c})$.  For $i,j \in \{1,2,3\}$ let
   $B_{i,j}=\{b \in B \text { : } L(P_{X,c})(b)=\{i,j\}\}$.
   Suppose first that  for every $T \in \mathcal{T}$ there exists
 $i \in \{1,2,3\} \setminus L(T)$ such that  $X_{T,i}$ is small.
 Then by \eqref{smallguess} no vertex of $B(P_{X,c})$ is mixed on an edge of
 $W$,   and therefore
 $P_{X,c}$ is nice.

 Thus we may assume that there exist $\{i,j,k\}=\{1,2,3\}$ and
 $D \in \mathcal{T}$ such that $L(D)=k$ and both
$X_{D,i}$ and $X_{D,j}$ are big. By \eqref{bigsmall} 
both $Y_i$ and $Y_j$ are big.  Since $M=2^{|S|}(r+6)$ and
there are at most $2^{|S|}$ types in $S$, it follows that
there exist
$T_i,T_j \subseteq S(P)$ such that
     $U_i \cap B(P,T_i) \geq r+6$, and
     $U_j \cap B(P,T_j) \geq r+6$.   
   We now show that $P_{X,c}$ is easy. We may assume that $i=1$ and $j=2$,
   and that $G|(B \cup W)$ contains a
   six-vertex path $R=p_1-p_2-p_3-p_4-p_5-p_6$.

Let $Y_1=\{a_1-b_1-c_1, \ldots, a_M-b_M-c_M\}$ and $Y_2=\{x_1-y_1-z_1, \ldots, x_M-y_M-z_M\}$. 
   We may assume that   $U_1 \cap B(P,T_1) =\{a_1, \ldots, a_{r+6}\}$ and  
    $U_2  \cap B(P,T_2)=\{x_1, \ldots, x_{r+6}\}$.
   Let $Y_1'=\{a_1-b_1-c_1, \ldots, a_{r+6}-b_{r+6}-c_{r+6}\}$ and let
   $Y_2'=\{x_1-y_1-z_1, \ldots, x_{r+6}-y_{r+6}-z_{r+6}\}$.

First we show that if $D' \in \mathcal{T}$ and 
$L(D') \neq 1$, and $X_{D',1}$ is big, then $D' \subseteq T_1$.
Suppose that there exists $s \in D'  \setminus T_1$.
Let $X_{D',1}=\{s_1-t_1-r_1, s_{2}-t_{2}-r_{2}\}$
   where $U_{D',1}=\{s_1,  s_2\}$.
Then $R'=t_2-s_2-s-s_1-t_1-r_1$ is a $P_6$.
Since $s \in S(P)  \setminus T_1$, it follows that $s$ has no neighbors in
the seagulls of $Y_1'$.
   By \eqref{Q}.1  $\{t_1,t_2,r_1\}$ is anticomplete to $V_1$.
   By \eqref{Q}.2 $\{s_1,s_2\}$ is anticomplete to    $U_1$.
   By \eqref{Q}.3      each of  $s_1,s_2$ has neighbors in at most one seagull
   of $Y_1'$, and thus
   $V(R')$ is anticomplete to at least
   $r+4$ seagulls of $Y_1'$, contrary to the fact that $G$ is $(P_6+rP_3)$-free.
   This proves that   $D'  \subseteq T_1$.

   By the claim of the previous paragraph with $D'=D$ and $i=1,2$,
   we deduce that $D \subseteq T_1 \cap T_2$. Consequently,
   $L(T_1)=L(T_2)=L(D)=3$.
   
 Recall that
   $V(Y_1) \cup V_{D,1} \cup V_{D,2} \subseteq S(P_{X,c})$. Therefore
   $V(R) \cap W$ is anticomplete to $V(Y_1) \cup V_{D,1} \cup V_{D,2}$
   $V(R) \cap (B_{12} \cup B_{13})$ is anticomplete
   to $U_1 \cup U_{T,1}$ and  $V(R) \cap (B_{12} \cup B_{23})$
   is anticomplete to $U_2 \cup U_{T,2}$. 
   
   By \eqref{Q}.3,
      every vertex of $B_{12} \cup B_{13}$  has neighbors in
   at most one seagull of $Y_1'$, and every vertex of
   $B_{12} \cup B_{23}$ has neighbors  in at most
   one seagull of $Y_2'$. If every vertex of $V(R)$ has
   neighbors in at most one of the seagulls of $Y_2'$, then 
   at least $|Y_2'|-6 \geq r$ of the seagulls in
   $Y_2'$ are anticomplete to $V(R)$, contrary to the fact that $G$
   is $(P_6+rP_3)$-free. This proves that some vertex $p_q \in V(R)$
   has neighbors in at least two of the seagulls of $Y_2'$.   It follows
   that $V(R) \cap B_{13} \neq \emptyset$, and we may assume
   $p_q \in V(R) \cap B_{13}$ has a neighbor in
   $x_1-y_1-z_1$ and in $x_2-y_2-z_2$.
   Since $c(x_1)=c(x_2)=2$,
   it follows that $c(y_1) \neq 2$ and $c(y_2) \neq 2$,  and so
   since $p_q \in B_{13}$, we deduce that $p_q$ is anticomplete
   to $\{y_1, y_2\}$.

   We claim that $p_q$ is not mixed on either of the the sets $\{y_1,z_1\}$,
   $\{y_2,z_2\}$.
   If $p_q \not \in B(P)$, this follows immediately from \ref{W}.
   Thus we may assume that $p_q \in B(P)$.
   Let $T'=N(p_q) \cap S(P)$.
   Since $c(T') \neq 3$, it follows that $T' \not \subseteq T_1$,
   and therefore   $X_{T',1}$ is small. Now the claim follows from
   \eqref{smallguess}.
   
   We deduce that
   $p_q$ is adjacent to $x_1,x_2$ and anticomplete to $\{y_1,z_1,y_2,z_2\}$.
   Now   $R'=z_1-y_1-x_1-p_q-x_2-y_2$ is a six-vertex path.
   By \eqref{acceptable}
   $\{x_1,y_1,z_1,x_2,y_2\}$ is anticomplete
   to $U_{1}$, $\{y_1,z_1,y_2\}$ is anticomplete to $V(Y_1)$,
   and each of $x_1,x_2$ has neighbors in at most one seagull of
   $Y_1'$.
   Since $p_q \in B_{13}$, \eqref{Q} implies that $p_q$ is anticomplete to $U_1$,
   and $p_q$ has neighbors in at most one  seagull in $Y_1'$. But now
   $R'$ is anticomplete to at least $|Y_1'|-3>r$ of the seagulls of $Y_1'$,
   contrary to the fact that $G$ is $(P_6+rP_3)$-free.
   This proves~\eqref{niceoreasy}.

   \begin{equation}
     \label{Ssize}
     |S(P_{X,c})| \leq |S|+3 \times \max (3M, f(M))+2^{|S|+1} \times  \max(6,f(2))
   \end{equation}
   First observe that $S(Q_{Y,d}) \setminus S=V(Y_1) \cup V(Y_2) \cup V (Y_3)$,
   and for every $i$, $|V(Y_i)| \leq \max(3M, f(M))$.
This implies that $|S(Q_{Y,d})| \leq |S|+3 \times \max (3M, f(M))$.
   The number of possible  pairs $(T,i)$ where
 $T \in \mathcal{T}$ and $i \in \{1,2,3\} \setminus L(T)$ is 
 $2t \leq 2^{|S|+1}$. For every such $(T,i)$,
   $|V_{T,i}| \leq \max(6,f(2))$, and therefore
   $|V(X)| \leq 2^{|S|+1} \times  \max(6,f(2))$.
   Since $S(P_{X,c})=S(Q_{Y,d}) \cup V(X)$, \eqref{Ssize} follows.
 
\bigskip

   Let  $\mathcal{L}(Q_{Y,d})$ be the collection of all $r$-seeded precolorings $P_{X,c}$
   where $X \in \mathcal{C}$ and $c$ is a consistent precoloring
   of $(G|V(X),L)$.
   
 \begin{equation}
   \label{L(P')size}
\longbox{   $|\mathcal{L}(Q_{Y,d})|$ is polynomial.} 
 \end{equation}
 The  number of possible  pairs $(T,i)$ where
 $T \in \mathcal{T}$ and $i \in \{1,2,3\} \setminus L(T)$ is
 $2t \leq 2^{|S|+1}$. For every such $(T,i)$,
 $|V_{T,i}| \leq \max(6,f(2))$, so there is a constant $C$
 that depends on $r$ but not on $G$, such that
 $|V_{T,i}| \leq C$. Consequently, for every
 $X \in \mathcal{C}$, $|V(X)| \leq 2tC$. It follows that
 $|\mathcal{C}| \leq |V(G)|^{2tC}$. Moreover, for every
 $X \in \mathcal{C}$, the number of precolorings of $(G|V(X),L)$ is at most
 $3^{|V(X)|} \leq 3^{2tC}$.
 Since $|\mathcal{L}(Q_{Y,d})|$ is at most the total number of pairs
 $(X,c)$  where  $X \in \mathcal{C}$ and $c$ is a precoloring
 of $(G|V(X),L)$, we deduce that
 $$|\mathcal{L}(Q_{Y,d})| \leq 3^{2tC}|\mathcal{C}| \leq 3^{2tC}|V(G)|^{2tC}.$$
 This proves~\eqref{L(P')size}.

 \bigskip
 
 Finally, let $\mathcal{L}=\bigcup_{P' \in \mathcal{M}}\mathcal{L}(P')$.

 \begin{equation}
   \label{Lsize}
   \longbox{$|\mathcal{L}|$ is polynomial}
 \end{equation}
 By \eqref{L(P')size} is is enough to prove that $|\mathcal{M}|$ is polynomial.
 We use the notation from the construction step of precolorings in
 $\mathcal{M}$. Let $C=\max(3M,f(M))$. 
Since $|V(X_i)| \leq \max(3M,f(M))=C$ for every $i$,
 the number of possible choices of $V(X)$ is at most $|V(G)|^{3C}$.
 The number of consistent colorings for of a given $X$ is at most
 $3^{|X|} \leq 3^{3C}$. It follows that
 $\mathcal{M} \leq (3|V(G)|)^{3C}$, as  required. This proves \eqref{Lsize}.

 \bigskip
 
 By \eqref{niceoreasy}, \eqref{Ssize}  and \eqref{Lsize}, it remains
 to show that $\mathcal{L}$ is equivalent to $P$. Since for
 every $P_{X,c} \in \mathcal{L}$, $c$ is a precoloring of $(G,L)$,
 it is clear that if some $P_{X,c}$ has a precoloring extension, then so
 does $P$. It remains to show that if $d$ is a precoloring extension of
 $P$, then some $R \in \mathcal{L}$ has a precoloring extension.

 Let $d$ be a precoloring extension of $P$. First we construct
 $Q \in \mathcal{M}$ that has a precoloring extension.
 Let $X_1=char_M(P,1,d)$. Then $X_1 \in guess(P,1)$ and $d$ is a $1$-consistent coloring of $V(X_1)$.
 Define the $r$-seeded precoloring $P(X_1,d)$ as follows. Let
 $\tilde{P}=(G,L',S')$ be obtained from $P$ by moving $V(X_1)$ to the seed with
 $d$. Next we modify $L'$ further.
   \begin{packed_itemize}
 \item Assume first that $X_{1}$ is small.
   If $b \in B(\tilde{P},1) \cap B(P)$ and
      $b$ is  mixed on an edge
   of $W(\tilde{P})$, remove $1$ from  $L'(b)$.
 \item Next assume that $X_{1}$ is big. Let
   $X_{1}=\{a_1-b_1-c_1, \ldots, a_{M}-b_{M}-c_{M}\}$ where
   $U_1=\{a_1, \ldots, a_M\} \subseteq B(P,1)$.
   If $b \in B(\tilde{P},1) \cap B(P)$ has
   a neighbor in $\{b_q,c_q\}$ for more than one value of $q$, remove $1$
   from $L'(b)$.
   \end{packed_itemize}
   Denote the precoloring we have constructed so far by $P(X_1,d)$.

We claim that $d(v) \in L(P(X_1,d))(v)$ for every $v \in V(G)$.
Suppose not. Since $\tilde{P}$ is obtained from $P$ by moving a set of vertices
to the seed with $d$, it follows that $d(v) \in L'(v)$ for every $v \in V(G)$.
Thus we may assume that for some $v \in V(G)$, $d(v) \in L'(v) \setminus L(P(X_1,d))(v)$. Suppose first that $X_1$ is small. Then
$v \in  B(\tilde{P},1) \cap B(P)$, $v$ is  mixed on an edge $w_1w_2$
of $G|W(\tilde{P})$, and $d(v)=1$. Then $w_1,w_2 \in W(P)$ and
by \ref{neatlemma}.1, at least one of $w_1,w_2$ has a neighbor in $X_1$.
It follows that not both $w_1,w_2$ are in $W(\tilde{P})$, a contradiction.
Thus we may assume that $X_1$ is big, $v \in B(\tilde{P},1) \cap B(P)$, $v$ has
a neighbor in $\{b_q,c_q\}$ for more than one value of $q$, and $d(v)=1$.
But this immediately contradicts \ref{neatlemma}.2.  This proves that
$d(v) \in L(P(X_1,d))(v)$ for every $v \in V(G)$.
   
Next let $X_2=char_M(P(X_1,d),2,d)$. Then $X_2 \in guess(P(X_1,d),2)$ and
$d$ is a $2$-consistent coloring of $V(X_2)$.
  Define the $r$-seeded precoloring $P(X_2,d)$ as follows. Let
$\tilde{P}=(G,L',S')$ be obtained from $P(X_1,d)$ by moving $V(X_2)$ to the seed
with $d$. Next we modify $L'$ further.
   \begin{packed_itemize}
 \item Assume first that $X_{2}$ is small.
   If $b \in B(\tilde{P},2) \cap B(P(X_1,d))$ and
      $b$ is  mixed on an edge
   of $W(\tilde{P})$, remove $2$ from  $L'(b)$.
 \item Next assume that $X_{2}$ is big. Let
   $X_{2}=\{a_1-b_1-c_1, \ldots, a_{M}-b_{M}-c_{M}\}$
   where $U_2=\{a_1, \ldots, a_M\} \subseteq B(P,2)$.
   If $b \in B(\tilde{P},2) \cap B(P(X_1,d))$ has
   a neighbor in $\{b_q,c_q\}$ for more than one value of $q$, remove $2$
   from $L'(b)$.
   \end{packed_itemize}
   Denote the precoloring we have constructed so far by $P(X_2,d)$.
   Repeating the previous  argument with $X_1$ replaced by $X_2$, we deduce that
   $d(v) \in L(P(X_2,d))$ for every $v \in V(G)$.
   Finally let $X_3=char_M(P(X_2,d),3,d)$.
Then $X_3 \in guess(P(X_2,d),3)$ and
$d$ is a $3$-consistent coloring of $V(X_3)$.
  Define the $r$-seeded precoloring $P(X_3,d)$ as follows. Let
$\tilde{P}=(G,L',S')$ be obtained from $P(X_2,d)$ by moving $V(X_3)$ to the seed
with $d$. Next we modify $L'$ further.
   \begin{packed_itemize}
 \item Assume first that $X_{3}$ is small.
   If $b \in B(\tilde{P},3) \cap B(P(X_2,d))$ and
      $b$ is  mixed on an edge
   of $W(\tilde{P})$, remove $3$ from  $L'(b)$.
 \item Next assume that $X_{3}$ is big. Let
   $X_{3}=\{a_1-b_1-c_1, \ldots, a_{M}-b_{M}-c_{M}\}$
   where $U_3=\{a_1, \ldots, a_M\} \subseteq B(P,3)$.
   If $b \in B(\tilde{P},3) \cap B(P(X_2,d) )$ has
   a neighbor in $\{b_q,c_q\}$ for more than one value of $q$, remove $3$
   from $L'(b)$.
   \end{packed_itemize}
   Denote the precoloring we have constructed by $P(X_3,d)$.
   Repeating the previous argument with $X_2$ replaced by $X_3$, we deduce that
   $d(v) \in L(P(X_3,d))$ for every $v \in V(G)$.
   Let $X=(X_1,X_2,X_3)$.
   Then $X \in \mathcal{Q}$
   and $d$ is a consistent precoloring of $V(X)$. 
   Let $Q_{Y,d}=P(X_3,d)$; then $Q_{Y,d} \in \mathcal{M}$ and $d$ is a precoloring
   extension of $Q_{Y,d}$.

   Now we construct $R \in \mathcal{L}(Q_{Y,d})$ that has a precoloring
   extension. For every $T \in \mathcal{T}$ and $i \in \{1,2,3\}$ let
   $X_{T,i}=char_2(Q_{Y,d},T,i,d)$. Then $d$ is a consistent coloring of $X_{T,i}$.
   Let $Q'$ be obtained by moving $\bigcup_{T,i}V(X_{T,i})$
   to the seed with $d$; write $L'=L(Q')$. We modify $L'$ further.
   For $T \in \mathcal{T}$ and $i \in \{1,2,3\}$ proceed as follows.
   \begin{packed_itemize}
   \item  Assume $X_{T,i}$ is small. If $b \in B(Q') \cap B(Q_{Y,d},T)$ is mixed on an
     edge of $W(Q')$, remove $i$ from $L'(b)$.
     \end{packed_itemize}

   Denote the precoloring thus obtained by $P_{X,d}$. It follows from the
   construction process of $\mathcal{L}(Q_{Y,d})$ that
   $P_{X,d} \in \mathcal{L}(Q_{Y,d})$.

We claim that $d(v) \in L(P_{X,d})(v)$ for every $v \in V(G)$.
Suppose not. Since $Q'$ is obtained from $Q_{Y,d}$ by moving a set of vertices
to the seed with $d$, it follows that $d(v) \in L'(v)$ for every $v \in V(G)$.
Thus we may assume that for some $v \in V(G)$,
$d(v) \in L'(v) \setminus L(P_{X,d})(v)$. Then $v \in B(Q') \cap B(P,T)$,
 $X_{T,i}$ is small, $v$  is mixed on an
     edge $w_1w_2$ of $G|W(Q')$, and $d(v)=i$.
Then $w_1,w_2 \in W(P)$ and
by \ref{neatlemmatype}.1, at least one of $w_1,w_2$ has a neighbor in $X_{T,i}$.
It follows that not both $w_1,w_2$ are in $W(Q')$, a contradiction.
This proves that
$d(v) \in L(P_{X,d})(v)$ for every $v \in V(G)$. 
Consequently,  $d$
   is a precoloring    extension of $P_{X,d}$. This proves~\ref{generaltonice}.
\end{proof}

\section{From nice to stable} \label{sec:nicetostable}

In this section we show that in order to be able to test if a nice $r$-seeded
precoloring has a precoloring extension, it is enough to be able to answer the
same question for a more restricted kind of $r$-seeded precoloring, that
we call ``stable''. 

We start with a lemma.
Let $G$ be a graph
    and $L$ a list assignment for $G$. We say that $v \in V(G)$ is
    {\em connected} if $G|N(v)$ is connected.
Let $v \in V(G)$   be connected such that $G|N(v)$ is bipartite.
Let $(A_1,A_2)$ be the (unique) bipartition of $G|N(v)$.
We say that $G'$ is obtained from $G$ by {\em reducing $v$}
if $V(G')=(V(G) \setminus (\{v\} \cup N(v))) \cup \{a_1,a_2\}$,
$G'\setminus \{a_1,a_2\}=G \setminus (\{v\}\cup N(v))$,
$a_1a_2 \in E(G')$, and for $u \in V(G) \cap V(G')$ and
$i \in \{1,2\}$, $a_iu \in E(G')$ if
and only if (in $G$) $u$ has a neighbor in $A_i$.
We say that $(G',L')$ is obtained from $(G,L)$ by {\em reducing $v$}
if $G'$ is obtained from $G$ by reducing $v$, $L'(u)=L(u)$
for every $u \in V(G') \setminus \{a_1,a_2\}$, and for $i=1,2$,
$L(a_i)=\bigcap_{a \in A_i}L(a)$.

\begin{theorem}
  \label{connectednbd}
  Let $r$ be an integer and let $G$ be a $(P_6+rP_3)$-free graph. 
  Let $v \in V(G)$ be connected such that $G|N(v)$ is bipartite with (unique)
  bipartition
  $(A_1,A_2)$, and let $G'$ be obtained from $G$ by reducing $v$. Then $G'$
  is $(P_6+rP_3)$-free.     
\end{theorem}

  \begin{proof}
    Suppose $Q$ is an induced subgraph of $G'$ isomorphic to $P_6+rP_3$.
    Recall that $v$ is anticomplete to $V(G') \setminus \{a_1,a_2\}$ (in $G$).
Then $V(Q) \cap \{a_1,a_2\} \neq \emptyset$.
    If only one vertex of $V(Q) \setminus \{a_1,a_2\}$, say $q$,  has a
    neighbor in $V(Q) \cap \{a_1,a_2\}$, say,
    $a_1$,    then we get a $P_6+rP_3$ in $G$ by replacing $a_1$ with
    a vertex of $N_G(q) \cap A_1$,
    and, if $a_2 \in V(Q)$, replacing $a_2$ with $v$.
Thus we may    assume that
    two vertices $q,q'$ of $V(Q) \setminus \{a_1,a_2\}$ have a neighbor in
    $V(Q) \cap \{a_1,a_2\}$. If
    $q$
    and $q'$ have a
    common neighbor $u \in A_1 \cup A_2$, then
    $G|((V(Q) \setminus \{a_1,a_2\}) \cup \{u\})$
    is a $P_6+rP_3$, a contradiction. So no such $u$ exists. 
Let $Q'$ be an induced path from $q$ to 
$q'$ with $V(Q') \setminus \{q,q'\} \subseteq A_1 \cup A_2 \cup \{v\}$,
meeting only one of $A_1,A_2$, if possible.
Then  $G|((V(Q) \setminus \{a_1,a_2\}) \cup V(Q'))$ contains 
a $P_6+rP_3$, a contradiction. This proves \ref{connectednbd}.
  \end{proof}

A $r$-seeded precoloring $P=(G,L,S)$ is {\em stable} if
\begin{packed_itemize}
  \item $P$ is nice.
\item Every component $C$ of $W(P)$ such that some $w \in C$ has $|L(w)|=3$
  satisfies $C=\{w\}$
\item Let $\{i,j,k\}=\{1,2,3\}$. Then for every $b \in B(P)_{ij}$
  the set $N(b) \cap B(P)_{ik}$ is stable.
\item Let $\{i,j,k\}=\{1,2,3\}$, let $w \in W(P)$ with $|L(w)|=3$ and let
  $n \in B(P)_{ij}$ and $n' \in B(P)_{jk}$ be adjacent to $w$. 
Then no $u \in B_{ik}$ is complete to $\{n,n'\}$.
\item No $w \in W(P)$ with $|L(w)|=3$ is connected, and
  \item  $deg(v)>2$ for every $v \in V(G)$ with $|L(v)|=3$.
\end{packed_itemize}

We can now prove the main result of this section.

\begin{theorem}
  \label{nicetostable}
  For every integer  $r>0$ there exists $c \in \mathbb{N}$ with the following
  properties.
  Let $G$ be a $(P_6+rP_3)$-free graph. Let $P=(G,L,S)$ be a nice
  $r$-seeded precoloring of $G$. Assume that
  for every $X \subseteq V(G)$ with $|X| \leq 4r+8$, the pair
  $(G|X,L)$ is  colorable. Then there exists a  collection
  $\mathcal{L}$ of stable $r$-seeded precolorings such that
  \begin{packed_enum}
    \item $|V(G(P'))| \leq |V(G)|$ for every $P' \in \mathcal{L}$, 
\item $S(P')=S(P)$ for every $P' \in \mathcal{L}$
    \item $\mathcal{L} \leq |V(G)|$, and 
    \item If for every $P' \in \mathcal{L}$ we know if $P'$ has a precoloring
      extension, then we can decide in polynomial time if $P$ has a precoloring
      extension, and construct one if it exists.
  \end{packed_enum}
  Moreover, $\mathcal{L}$ can be constructed in time $O(|V(G)|^c)$.
\end{theorem}

\begin{proof}
  In the proof we describe several modifications that can be made to $P$
  (in polynomial time)
  without changing the existence of a precoloring extension.

  \begin{equation}
    \label{reducev}
    \longbox{Let $v \in V(G)$ with $|L(v)|=3$ 
    such that either $deg_G(v) \leq 2$, or $v$ is connected. Then 
    we can construct in polynomial time a $(P_6+rP_3)$-free graph $G'$ 
    and an $r$-seeded precoloring $P'=(G',L',S)$ such that
    $|V(G')|<|V(G)|$, and $\{P'\}$ is equivalent to $P$.}
      \end{equation}
      If $deg_G(v)<2$ we can set $(G',L',S)=(G\setminus v, L,S)$;
      thus we may assume that $v$ is connected. If $G|N(v)$ is not bipartite, then $G|N(v)$ contains an odd cycle $C$ (as an induced subgraph). Since $G$ is $(P_6+rP_3)$-free, and therefore $P_{6+4r}$-free, it follows that $|V(C)| \leq 4r+7$.
      But now $G|(V(C) \cup \{v\})$ is not $3$-colorable, and therefore
      is not $L$-colorable, a contradiction. So we may assume
      that $G|N(v)$ is bipartite with bipartition $(A_1,A_2)$.
      Since $G|N(v)$ is connected, it follows that the bipartition is unique.
      Let $(G',L')$ be obtained from $(G,L)$ by reducing $v$.
      Since $|L(v)|=3$, it follows that $v \in W(P)$, and therefore
      $(A_1 \cup A_2) \cap S = \emptyset$. Thus
      $P'=(G',L',S)$ is an $r$-seeded precoloring of $G'$.
      By     \ref{connectednbd} $G'$ is $(P_6+rP_3)$-free. The
      uniqueness of the bipartition $(A_1,A_2)$ implies that in every
      coloring of
      $(G,L)$  each of the sets $A_1,A_2$ is monochromatic. This in turn
      implies
      that if $c$ is a precoloring extension of $P$, then a precoloring
      extension $c'$ of
      $P'$ can be obtained by setting $c'(u)=c(u)$ for every $u \in V(G') \setminus \{a_1,a_2\}$, and $c'(a_i)=c(A_i)$ for $i=1,2$. Conversely, if $c'$ is
      a precoloring extension of $P'$, then 
    a precoloring extension $c$ of $P$ can be obtained by setting
      $c(u)=c'(u)$ for every $u \in V(G) \cap V(G')$, $c(a)=c'(a_i)$
      for every $a \in A_i$, and $c(v)=\{1,2,3\} \setminus \{c'(a_1),c'(a_2)\}$.
      This proves~\eqref{reducev}.

      \bigskip

      \begin{equation}
        \label{stableboundary}
        \longbox{Let $\{i,j,k\}=\{1,2,3\}$. Suppose that  $b \in B(P)_{ij}$
          has neighbors $n,n' \in B(P)_{ik}$ such that $n$ is adjacent to $n'$.
          Let $P'=(G',L',S)$ be 
          the $r$-seeded precoloring obtained by setting
          $L'(b)=\{j\}$, and
          $L'(v)=L(v)$ for every $v \in V(G) \setminus \{b\}$.
          Then          $\{P'\}$ is equivalent to $P$.}
      \end{equation}
      \eqref{stableboundary} follows from the fact that in every precoloring
      extension of $P$ one of $n,n'$ receives color $i$.

      \bigskip

       \begin{equation}
        \label{nosquare}
        \longbox{Let $\{i,j,k\}=\{1,2,3\}$, let $w \in W(P)$ with $|L(w)|=3$
          and let
  $n \in B(P)_{ij}$ and $n' \in B(P)_{jk}$ be adjacent to $w$. 
Suppose that some  $u \in B(P)_{ik}$ is complete to $\{n,n'\}$.
          Let $P'$ be 
          the seeded precoloring $(G,L',S)$ obtained by setting
          $L'(w)=\{1,2,3\} \setminus \{j\}$. Then
          $\{P'\}$ is equivalent to $P$.}
      \end{equation}
       Clearly a precoloring extension of $P'$ is also a precoloring extension of $P$.
       To see the converse, let $c$ be a precoloring extension of $P$.
       We may assume by symmetry that $c(u)=i$. Then $c(n)=j$, and
       so $c(w) \neq j$, and $c$ is a precoloring extension of $P'$. This
       proves \eqref{nosquare}.

      \bigskip

      Repeatedly applying \eqref{reducev}, \eqref{stableboundary} and
      \eqref{nosquare}
      we may assume that

      \begin{equation}
        \label{almoststable}
\longbox{\begin{packed_itemize}
\item No $w \in W$ with $|L(w)|=3$ is connected, 
  \item  $deg(v)>2$ for every $v \in V(G)$ with $|L(v)|=3$.
  \item For every $\{i,j,k\}=\{1,2,3\}$ and  for every $b \in B(P)_{ij}$, the set      $N(b) \cap B(P)_{ik}$ is stable.
\item Let $\{i,j,k\}=\{1,2,3\}$, let $w \in W(P)$ with $|L(w)|=3$ and let
  $n \in B(P)_{ij}$ and $n' \in B(P)_{jk}$ be adjacent to $w$. 
Then no   $u \in B(P)_{ik}$ is complete to $\{n,n'\}$.
\end{packed_itemize}}
      \end{equation}

\begin{equation}
    \label{X0cutset}
    \longbox{We may assume that $G \setminus X^0(L)$ is connected.}
    \end{equation}
  Suppose not. Let $C_1, \ldots, C_m$ be the components of $G \setminus X^0(L)$.
  For $i \in \{1, \ldots, m\}$ let $G_i=G|(X^0(L) \cup C_i)$. Then
  $S \subseteq V(G_i)$ for every $i$. Moreover,
  $P_i=(G_i,S,L)$ is an $r$-seeded precoloring of $G_i$, where
  $B(P_i)=B(P) \cap C_i$ and $W(P_i)=W(P) \cap C_i$. It follows
  that each of $P_i$ is nice, and satisfies \eqref{almoststable}.
  Clearly if $P$ has a precoloring
  extension, then each $P_i$ does. Conversely, if each $P_i$ has a
  precoloring extension $c_i$, then setting $c(v)=c_i(v)$ for
  $v \in V(G_i)$, we obtain a precoloring extension of $P$.
  Now it is enough to prove the
  theorem for each $P_i$ separately, and \eqref{X0cutset} follows.

  \bigskip

  In view of \eqref{X0cutset} from now on we assume that
  $G \setminus X^0(L)$ is connected.    
      It remains to show that:
      
      \begin{equation}
\label{singletoncomp}
\longbox{If $C$ is a component of $W(P)$ and $w \in C$ has $|L(w)|=3$,
        then $W=\{c\}$.}
\end{equation}
      Let $C$ and $w$ be as above. Since $G \setminus X^0(L)$ is connected,
      it follows that $B(P)(C)$ (this is the set of attachments of $C$ in
      $B(P)$) is non-empty and complete to $C$. Since $|L(w)|=3$, it
      follows from the definition of a seeded precoloring that $w$ is
      anticomplete to $X^0(L)$, and consequently
      $N(w)=(B(P)(C)) \cup (N(w) \cap C )$. If $C \neq \{w\}$, then
      $N(w) \cap C \neq \emptyset$, and therefore $w$ is connected,
      a contradiction. This proves~\eqref{singletoncomp}.

      \bigskip
      
      Now \ref{nicetostable} follows from \eqref{almoststable} and \eqref{singletoncomp}.
      \end{proof}
  
\section{Reducing lists} \label{sec:stableto2SAT}

The goal of this section is to deal with stable precolorings.
Similarly to Section~\ref{sec:generaltonice}, will define several
``characteristics'' of a precoloring extension of an $r$-seeded precoloring,
and then, in the algorithm, given an $r$-seeded precoloring, enumerate
all possible characteristics of its precoloring extensions.

First we need a few more definitions. Let $P(G,L,S)$ be an $r$-seeded precoloring.
We write $\tilde{W}(P)=\{w \in W(G) \; : \: |L(w)|=3\}$ and denote by
$\tilde{B}(P)$ the set of attachments of $\tilde{W}(P)$ in $B(P)$.
We write $\tilde{B}(P,i)=B(P,i) \cap \tilde{B}(P)$,
$\tilde{B}(P)_{ij}=B(P)_{ij} \cap \tilde{B}(P)$ and
$\tilde{B}(P,T)=B(P,T) \cap \tilde{B}(P)$.

Let $P=(G,S,L)$ be an $r$-seeded precoloring of a $(P_6+rP_3)$-free graph $G$,
and let $c$ be a precoloring extension of $P$. 
We define several hypergraphs associated with $P$ and $c$.
For every $i  \in \{1,2,3 \}$, let
$V_i= \{b \in \tilde{B}(P,i) \text{ : } c(b)=i\}$.

  For every distinct $i,j \in \{1,2,3\}$    we define the hypergraph $R(P,i,j,c)$
 with vertex set $V_i$ as follows.
 Let $K$  be the set of all vertices $w \in \tilde{W}(P)$
 such that
$w$ has two neighbors $n,n' \in \tilde{B}(P)_{ij}$ with $c(n)=c(n')=i$. Let
$h(w)=N(w) \cap V_i$.
Then $\{h(w) \text{ : } w \in K  \}$ is the set of the hyperedges of
$R(P,i,j,c)$.

We prove:

\begin{theorem}
  \label{RLemma}
  There is a function $f:\mathbb{N} \rightarrow \mathbb{N}$ with the following
  properties. 
  Let $r>0$ be an integer, $G$ a $(P_6+rP_3)$-free graph with no clique of size four,  $P=(G,L,S)$
  a stable $r$-seeded precoloring, and let $c$ be a coloring of $(G,L)$.
  Then for every integer $M$ and 
  every distinct $i, j \in  \{1,2,3\}$ either
  \begin{packed_enum}
  \item there exists $X \subseteq V(R(P,i,j,c))$ with $|X| \leq f(M)$ such that
    if $w \in \tilde{W}(P)$ is anticomplete to $X$, then $w$ has at most one
     neighbor $n \in \tilde{B}(P)_{ij}$ with 
    $c(n)=i$, or 
      \item there exists a flock $F=\{a_1-b_1-c_1, \ldots, a_{M}-b_{M}-c_{M}\}$
    where for every $l$, $a_l,c_l \in V_i \cap \tilde{B}(P)_{ij}$ and
    $b_l \in \tilde{W}(P)$ ,and 
    such that every vertex of $V_i$  is adjacent to $b_l$
    for at most one value of $l$.
  \end{packed_enum}
\end{theorem}

  \begin{proof}
    Let $f=f_{r,1,1}$  be as in \ref{nutau}.
    Since $P$ is stable, $\tilde{W}(P)$ is a stable set, and
    $\tilde{W}(P)$ is anticomplete to $W(P) \setminus \tilde{W}(P)$.
  Applying \ref{nutau} to $H=R(P,i,j,c)$ with $p=q=1$, we deduce that
  either $\nu(H) \geq M$ or $\tau(H) \leq f(M)$.

  Suppose first  that $\nu(H) \geq M$. Let $b_1, \ldots, b_M \in \tilde{W}(P)$ be
  such that $M=\{h(b_1), \ldots, h(b_{M})\}$ is a
  matching of $H$.
  It follows from the definition of $h(b_l)$ that
  for every $l$ there exists $a_l,c_l \in \tilde{B}(P)_{ij}$ such that
$c(a_l)=c(c_l)=i$, and consequently 
  $a_l-b_l-c_l$ is a
  seagull. Moreover, since $M$ is a matching of $H$, no $v \in V_i$
  belongs to more than one $h(b_l)$, and therefore every vertex of
  $V_i$  is adjacent to $b_l$ 
  for at most one value of $l$. This proves that  if $\nu(H) \geq M$, then
  \ref{RLemma}.2 holds.
  Thus we may assume that $\tau(H) \leq f(M)$. Letting
  $X \subseteq V_i$ be a hitting set for $H$, we immediately see that
  \ref{RLemma}.1 holds.
\end{proof}

Given an $r$-seeded precoloring $P$,
distinct $i, j \in \{1,2,3\}$ and  a precoloring
extension $c$ of $P$,
we say that an {\em $R,M$-characteristic}  of $P,i,j,c$ (denoted by
$char_{R,M}(P,i,j,c)$) is $X$ if
\ref{RLemma}.1 holds for $P,i,j$  and $c$,  and $F$ if
\ref{RLemma}.2
holds for $P,i,j$ and $c$.  We denote by $V(char_{R,M}(P,i,j,c))$
the set of all the vertices involved in $char_{R,M}(P,i,j,c)$.

Next, for every $i \in \{1,2,3\}$,  we define another hypergraph, $S(P,i,c)$,
with  vertex set $V_i$.
Let $K$  be the set of all vertices $w \in \tilde{W}(P)$ such that
$w$ has a neighbor $n \in \tilde{B}(P)_{ij}$ and $n' \in \tilde{B}(P)_{ik}$
  with $c(n)=c(n')=i$. 
    Let $h(w)=N(w) \cap V_i$.
Then $\{h(w) \text{ : } w \in K  \}$ is the set of the hyperedges of
$S(P,i,c)$.

We prove an analogue of \ref{RLemma}.

\begin{theorem}
  \label{SLemma}
  There is a function $f:\mathbb{N} \rightarrow \mathbb{N}$ with the following
  properties. 
  Let $r>0$ be an integer, $G$ a $(P_6+rP_3)$-free graph with no clique of size four,  $P=(G,L,S)$
  a stable $r$-seeded precoloring, and let $c$ be a coloring of $(G,L)$.
  Then for every integer $M$ and 
  every $i \in \{1,2,3\}$ either
  \begin{packed_enum}
  \item there exists $X \subseteq V_i$ with $|X| \leq f(M)$ such that
    if $w \in \tilde{W}(P)$ is anticomplete to $X$, then 
either $c(N(w) \cap \tilde{B}(P)_{ij})=j$ or $c(N(w) \cap \tilde{B}(P)_{ik})=k$, or
      \item there exists a flock $F=\{a_1-b_1-c_1, \ldots, a_{M}-b_{M}-c_{M}\}$
        where for every $l$,
        $a_l \in \tilde{B}(P)_{ij}, b_l \in \tilde{W}(P)$ and $ c_l \in \tilde{B}(P)_{ik}$, $c(a_l) = c(c_l)=i$,
        and 
           such that every vertex of $V_i$  is adjacent to $b_l$
    for at most one value of $l$.
  \end{packed_enum}
\end{theorem}

\begin{proof}
 Since $P$ is stable, $\tilde{W}(P)$ is a stable set, and
    $\tilde{W}(P)$ is anticomplete to $W(P) \setminus \tilde{W}(P)$.
  Let $f=f_{r,1,1}$  be as in \ref{nutau}.
  Applying \ref{nutau} to $H=S(P,i,c)$ with $p=q=1$, we deduce that
  either $\nu(H) \geq M$ or $\tau(H) \leq f(M)$.

  Suppose first  that $\nu(H) \geq M$. Let $b_1, \ldots, b_M \in \tilde{W}(P)$ be
  such that $M=\{h(b_1), \ldots, h(b_{M})\}$ is a
  matching of $H$.
  It follows from the definition of $h(b_l)$ that
  for every $l$ there exists $a_l \in \tilde{B}(P)_{ij}$ and $c_l \in \tilde{B}(P)_{ik}$
  with $c(a_l)=c(c_l)=i$, 
and consequently 
  $a_l-b_l-c_l$ is a
  seagull. Moreover, since $M$ is a matching of $H$, no $v \in V_i$
  belongs to more than one $h(b_l)$, and therefore every vertex of
  $V_i$  is adjacent to $b_l$ 
  for at most one value of $l$. This proves that  if $\nu(H) \geq M$, then
  \ref{SLemma}.2 holds.
  Thus we may assume that $\tau(H) \leq f(M)$. Letting
  $X \subseteq V_i$ be a hitting set for $H$, we immediately see that
  \ref{SLemma}.1 holds.
\end{proof}

Given an $r$-seeded precoloring $P$,
$i  \in \{1,2,3\}$ and  a precoloring
extension $c$ of $P$,
we say that an {\em $S,M$-characteristic}  of $P,i,c$ (denoted by
$char_{S,M}(P,i,c)$) is $X$ if
\ref{SLemma}.1 holds for $P,i$  and $c$,  and $F$ if
\ref{SLemma}.2
holds for $P,i$ and $c$.  We denote by $V(char_{S,M}(P,i,c))$
the set of all the vertices involved in $char_{S,M}(P,i,c)$.

We also need a version of the hypergraph above for types, as follows.
Let $\{i,j,k\}=\{1,2,3\}$.
For every pair of types  $T_1,T_2 \subseteq S$ with $c(T_1)=k$ and
$c(T_2)=j$, 
we define the hypergraph $H(P,T_1,T_2,c)$. The vertex set
$V(H(P,T_1,T_2,c))=\{b \in \tilde{B}(P,T_1) \cup \tilde{B}(P,T_2)  \text{ : } c(b)=i \}$. Next we
construct
the hyperedges. Let $K$  be set of all $w \in W$
such that $w$ has  neighbors $n \in \tilde{B}(P,T_1)$  and $n' \in \tilde{B}(P,T_2)$
with $c(n)=c(n')$ (and therefore $c(n)=c(n')=i$).
For every $w \in K$,  let $h(w)=N(w) \cap V(H(P,T_1,T_2,c))$.
Then
$\{h(w) \text{ : } w \in K  \}$ is the set of the hyperedges of
  $H(P,T_1,T_2,c)$.

\begin{theorem}
  \label{SLemmatypes}
  There is a function $f:\mathbb{N} \rightarrow \mathbb{N}$ with the following
  properties.
  Let $r>0$ be an integer, $G$ a $(P_6+rP_3)$-free graph with no clique of size four,  $P=(G,L,S)$
  a stable $r$-seeded precoloring, and let $c$ be a coloring of $(G,L)$.
  Let $\{i,j,k\}=\{1,2,3\}$.
  Then for every integer $M$ and for every pair of types
  $T_1,T_2$ of $S$ with $c(T_1)=k$ and $c(T_2)=j$
  either
  \begin{packed_enum}
  \item there exists $X \subseteq V(H(P,T_1,T_2,c))$ with
    $|X| \leq f(M)$ such that if $w \in W$ is anticomplete to $X$, then either
    $c(N(w) \cap \tilde{B}(P,T_1))=j$  or $c(N(w) \cap  \tilde{B}(P,T_2))=k$, or
  \item there exists a flock $F=\{a_1-b_1-c_1, \ldots, a_M-b_M-c_M\}$
    with $a_1, \ldots, a_M \in V(H(P,T_1,T_2,c)) \cap \tilde{B}(P,T_1)$,
$c_1, \ldots, c_M \in V(H(P,T_1,T_2,c)) \cap \tilde{B}(P,T_2)$,
    and
    $b_1, \ldots, b_M \in \tilde{W}(P)$.
  \end{packed_enum}
\end{theorem}

\begin{proof}
 Since $P$ is stable, $\tilde{W}(P)$ is a stable set, and
    $\tilde{W}(P)$ is anticomplete to $W(P) \setminus \tilde{W}(P)$.
  Let $f=f_{r,1,1}$  be as in \ref{nutau}.
    Applying \ref{nutau} to $H=H(P,T_1,T_2,c)$ with $p=q=1$, we deduce
    that
  either $\nu(H) \geq M$ or $\tau(H) \leq f(M)$.
  Suppose first  that $\nu(H) \geq M$. Let $b_1, \ldots, b_M \in \tilde{W}(P)$
  be such that $M=\{h(b_1), \ldots, h(b_M)\}$ is a matching of $H$.
  It follows from the definition of $h(b_l)$ that
  there exists a flock  as in \ref{SLemmatypes}.2.
  Thus we may assume that $\tau(H) \leq f(M)$. Letting
  $X \subseteq V(H(P,T_1,T_2,c))$ be a hitting set for $H$, we immediately see
  that
  \ref{SLemmatypes}.1 holds.
\end{proof}

Given an $r$-seeded precoloring $P$  and types $T_1,T_2$ of $S(P)$ with
$c(T_1) \neq C(T_2)$,
and  a precoloring extension $c$ of $P$,
we say that an {\em $M$-characteristic}  of $P,T_1,T_2,c$ (denoted by
$char_M(P,T_1,T_2,c)$) is $X$ if
\ref{SLemmatypes}.1 holds for $P,T_1,T_2$  and $c$,  and
$F$ if \ref{SLemmatypes}.2
holds for $P,T_1,T_2$ and $c$.  We denote by $V(char_M(P,T_1,T_2,c))$ the
set of all the vertices involved in $char_M(P,T_1,T_2,c)$.

In contrast to Section~\ref{sec:generaltonice}  here
we will need another type of characteristic, that is not related to
\ref{nutau}.
Let $\{i,j,k\}=\{1,2,3\}$. A seagull $a-b-d$ is an
{\em $ij$-typed seagull} if $a \in \tilde{B}(P)_{ik}$, $b \in \tilde{W}(P)$ and 
$d \in \tilde{B}(P)_{jk}$.
An $ij$-typed seagull is 
{\em $ij$-colored} if
$c(a)=i$ and $c(d)=j$
(and therefore $c(b)=k$).
Let $width_{ij}(c)$ be the maximum size of a flock $F$ of
$ij$-colored seagulls.
We say that two $ij$-typed seagulls $a-b-c$ and $a'-b'-c'$ are {\em related}
$a$ is adjacent to $c'$, $c$ is adjacent to $a'$,
and there are no other edges between $s$ and $s'$.
A {\em $(P,i,j,c)$-key} is a pair $(X_1,X_2)$ such that
\begin{packed_itemize}
\item $X_1$ is a maximal flock of $ij$-colored seagulls.
Let $X_1=\{x_1-y_1-z_1, \ldots, x_m-y_m-z_m\}$.
  Let $P'=(G,L',S')$ be obtained from $P$ by moving $V(X_1)$ to the seed
  with $c$.
\item For every $l \in \{1, \ldots, m\}$
  let $S_l$ be a flock of size at most one, such that
  if $S_l \neq \emptyset$, then the member of $S_l$ is an $ij$-typed seagull
  of $P'$ related to $x_l-y_l-z_l$. Let $X_2=\bigcup_{l=1}^m S_l$.
\item For every $s_2 \in X_2$, at least one wing of  $s_2$ has color $k$
  (in $c$).
\item For every $l \in \{1, \ldots, m\}$, if $S_l=\emptyset$, then no
  $ij$-typed seagull of $P'$ that is related to $s_l$ has a
  wing $u$ with $c(u)=k$.
  \end{packed_itemize}
The {\em order} of the key is $|X_1|$.

\begin{theorem}
  \label{keylemma}
  Let $r>0$ be an integer, let $G$ be a $(P_6+rP_3)$-free graph with no clique of
  size four,  and let $P=(G,L,S)$ be
  a stable $r$-seeded precoloring.
  Assume that $G|(\tilde{B}(P)_{ik} \cup \tilde{B}(P)_{jk} \cup \tilde{W}(P))$ is $P_6$-free.
  Let $F$ be a flock of  $ij$-typed seagulls in $P$.
  Let $c$ be a coloring of $G|V(F)$ where
  $c(V(F) \cap \tilde{B}(P)_{ik})=i$,  $c(V(F) \cap \tilde{B}(P)_{jk})=j$,
  $c(V(F) \cap \tilde{W}(P))=k$,
  and  let $P'$ be the precoloring obtained from $P$ by moving $V(F)$ to the seed with $c$. Then
  \begin{packed_enum}
  \item For every $ij$-typed seagull $s$  of $P'$,    either
    $F \cup \{s\}$ is a flock, or $s$ is related to a seagull of $F$.
  \item If $s \in F$, and $s_1=x_1-y_1-z_1$ and $s_2=x_2-y_2-z_2$  are
    $ij$-typed seagulls of $P'$ 
    such that  both $s_1$ and $s_2$ are related to $s$, and
       $y_1$ is anticomplete to
    $\{x_2,z_2\}$, and $y_2$ is anticomplete to $\{x_1,z_1\}$,
    then    $s_1$ is related to $s_2$.
  \end{packed_enum}
  \end{theorem}

\begin{proof}
  Let $F=\{a_1-b_1-c_1, \ldots, a_m-b_m-c_m\}$.
  Let $s=x-y-z$ be an $ij$-typed seagull of $P'$.
We may assume that $F \cup \{s\}$ is not a flock.
Since every seagull of $F$ is an $ij$-colored seagull in $c$, and  $s$
is an $ij$-typed seagull of $P'$,
and $V(F) \subseteq S(P')$, it follows that
for every $l \in \{1, \ldots, m\}$,
  the only possible edges between
$\{a_l,b_l,c_l\}$ and $\{x,y,z\}$ are $a_lz$ and $c_lx$.
  By symmetry we may assume that for some $l \in \{1, \ldots, m\}$
  $x$ is adjacent to $c_l$.
  Since $a_l-b_l-c_l-x-y-z$ is not a
  $P_6$ in  $G|(\tilde{B}(P)_{ik} \cup \tilde{B}(P)_{jk} \cup \tilde{W}(P))$, it
  follows that
  $a_l$ is adjacent to $z$, and thus $s$ is related to
  $a_l-b_l-c_l$. This proves the first assertion of \ref{keylemma}.

  We now prove the second assertion. Assume that $s_1=x_1-y_1-z_1$ and
  $s_2=x_2-y_2-z_2$ are both $ij$-typed seagulls of $P'$ that are related
  to $a_1-b_1-c_1$, and $y_1$ is anticomplete to
  $\{x_2,z_2\}$, and $y_2$ is anticomplete to $\{x_1,z_1\}$.
  Then $b_1,y_1,y_2 \in \tilde{W}(P)$. Since $P$ is stable,
  it follows that $x_1,z_1,x_2,z_2, a_1,c_1 \in \tilde{B}(P)$.
  This implies that $y_1$ is not adjacent to $y_2$.
  The fact that $c(c_1)=j$ implies  that $c_1$ is complete to
  $\{x_1,x_2\}$ and anticomplete to $\{z_1,z_2\}$.
Since $P$ is stable   and $c_1$ is complete
  to $\{x_1,x_2\}$, it follows that $x_1$ is non-adjacent to $x_2$,
  and similarly
  and $z_1$ is non-adjacent to $z_2$.  Since
  $y_1-x_1-c_1-x_2-y_2-z_2$ is not a $P_6$, it follows that
  $x_1$ is adjacent to $z_2$, and similarly $x_2$ is adjacent to $z_1$.
  This proves the second assertion of \ref{keylemma} and
  and completes the proof.
  \end{proof}

\begin{theorem}
  \label{heterolemma}
  Let $r>0$ be an integer, $G$ a $(P_6+rP_3)$-free graph with no clique of size four,  $P=(G,L,S)$
  a stable $r$-seeded precoloring, and let $c$ be a coloring of $(G,L)$.
Assume that $G|(\tilde{B}(P)_{ik} \cup \tilde{B}(P)_{jk} \cup \tilde{W}(P)$ is $P_6$-free. 
  Then for every integer $M$ and 
  every $\{i,j,k\}= \{1,2,3\}$ either
  \begin{packed_enum}
  \item $width_{i,j}(c) \geq  M$, or
  \item There exists a $(P,i,j,c)$-key $(X_1,X_2)$  of order less than $M$.
\end{packed_enum}
  \end{theorem}

\begin{proof}
  We may assume that \ref{heterolemma}.1 does not hold. Let $X_1$
  be a maximal flock of $ij$-colored seagulls; then
  $|X_1|=m < M$. Write $X_1=\{a_1-b_1-c_1, \ldots, a_m-b_m-c_m\}$.
  Note that for every $v \in V(X_1) \cap \tilde{W}(P)$ we have $c(v)=k$.
  Let $P'$ be the precoloring obtained from $P$ by moving $V(X_1)$ to the
  seed with $c$.
 For every $l$, let $S_l$ contain an $ij$-typed seagull $x-w-y$ of $P'$ that is related
 to $s_l$ and  that has $c(x)=k$ or $c(y)=k$. If no such $x-w-y$ exists,
 let $S_l= \emptyset$. Let $X_2=\bigcup_{l=1}^mS_l$.
 Clearly $X=X_1 \cup X_2$ is a $(P,i,j,c)$-key.
\end{proof}

Given an $r$-seeded precoloring $P$,
distinct $i,j \in \{1,2,3\}$ and  a precoloring
extension $c$ of $P$,
we say that a {\em heterogeneous $M$-characteristic}
of $P,i,j,c$ (denoted by
$char_{h,M}(P,i,j,c)$) is a flock of size $M$ of
$ij$-colored seagulls  if
\ref{heterolemma}.1 holds for $P,i,j$  and $c$,  and a $(P,i,j,c)$-key
$(X_1,X_2)$ of order $<M$ if
\ref{heterolemma}.2
holds for $P,i,j$ and $c$.  We denote by $V(char_{h,M}(P,i,j,c))$
the set of all the vertices involved in $char_{h,M}(P,i,j,c)$.

Next we generalize the notion of an $r$-seeded precoloring in order
to be able to use \ref{Mono}.
For a graph $G$, the pair $(P,\mathcal{X})$ is 
an {\em augmented $r$-seeded precoloring} of $G$ if
$P$ is an $r$-seeded precoloring of an induced subgraph $G'$ of $G$ and
$\mathcal{X}$ is a set
of subsets of $V(G)$ where $|\mathcal{X}|$ is polynomial.
A precoloring extension of $P$ is a {\em coloring} of $(P, \mathcal{X})$
if every $X \in \mathcal{X}$ is monochromatic in $c$.
We say that $(P, \mathcal{X})$ is {\em tractable} if
$|L(P)(v)| \leq 2$ for  every $v \in V(G')$.

For an $r$-seeded precoloring $P$, a collection  $\mathcal{L}$ of augmented
$r$-seeded precolorings is {\em equivalent} to $P$ if $P$ has a
precoloring extension if and only if some member of $\mathcal{L}$ has a
coloring, and given a coloring of a member of $\mathcal{L}$, a precoloring
extension of $P$ can be constructed in polynomial time.

We can now prove the main result of this section. The strategy of the proof is
similar to \ref{generaltonice}, but it is technically more involved since
we need to consider more characteristics.
\begin{theorem}
  \label{stableto2SAT}
   There exists a function $g: \mathbb{N} \rightarrow \mathbb{N}$
   with the following properties.
   Let $r>0$ be an integer, $G$ a $(P_6+rP_3)$-free graph with no clique of
   size four, and $P=(G,L,S)$ a stable $r$-seeded precoloring of $G$.
    Then there exists an equivalent collection
  $\mathcal{L}$ of tractable augmented $r$-seeded precolorings such that
  $|\mathcal{L}| \leq |V(G)|^{g(|S|)}$. 
  Moreover, $\mathcal{L}$ can be constructed in time $O(|V(G)|^{g(|S|)})$.
  \end{theorem}

\begin{proof}
  We start with an observation.
  
  \begin{equation}
    \label{safe}
    \longbox{Let $w \in \tilde{W}(P)$. Then $N(w) \subseteq \tilde{B}(P)$.}
      \end{equation}
  \eqref{safe}  follows immediately from the fact that $P$ is stable, and in
  particular nice.

  \bigskip

  Let $i, j \in \{1,2,3\}$ be distinct,
$f_R$  as in \ref{RLemma} and $M=r+6$.
First we enumerate $R,M$-characteristics of $P$.
  Let $smallguess(P,i,j)$ be the set of all subsets of
  $\tilde{B}(P,i)$ of size at most $f_R(M)$, and $bigguess(P,i,j)$  the set
  of all flocks of size $M$ such that every seagull of the flock
  has body in $\tilde{W}(P)$ and both its wings in $\tilde{B}(P)_{ij}$.
    Let $guess(P,i,j)=smallguess(P,i,j) \cup bigguess(P,i,j)$.
    Then $guess(P,i,j)$ is the set of all possible objects that can be an
    $R,M$-characteristic of a precoloring extension of $P$.
    We say that $X_{ij} \in guess(P,i,j)$ is {\em small} if
  $X_{ij}  \in smallguess(P,i,j)$ and that $X_{ij}$ is {\em  big} if
  $X_{ij}  \in bigguess(P,i,j)$. If $X_{ij}$ is big, we denote
by $U_{ij}$ the set of the wings of the flock, by $W_{ij}$ the the bodies of
the flock, and write
  $V_{ij}=U_{ij} \cup W_{ij}$.  If $X_{ij}$ is small, we write
  $V_{ij}=U_{ij}=X_{ij}$, and $W_{ij}=\emptyset$. 
In both cases we set $V(X_{ij})=V_{ij}$.
A precoloring $c$ of $(G|V_{ij}, L)$ is {\em $i,j$-consistent} if
  $c(v)=i$ for every $v \in U_{ij}$.

 Let $\mathcal{Q}$ be the set of all $6$-tuples $X=(X_{12},X_{21},X_{13},X_{31},X_{23},X_{32})$ such that $X_{ij} \in guess(P,i,j)$. Let
 $V(X)=\bigcup_{i \neq j \in \{1,2,3\}}V(X_{ij})$. We say that a precoloring $c$ of $(G|V(X), L)$
 is {\em consistent} if
 $c|V_{ij}$ is $i,j$-consistent for all $i,j$.
 
 Let  $X \in \mathcal{Q}$ and let $c$ be  a consistent precoloring  of $X$,
 and let $P'=(G,L',S')$ be the precoloring obtained from $P$ by moving $V(X)$ to the
 seed with $c$. We modify $L'$ further as follows. Let $i,j \in \{1,2,3\}$.

 \begin{packed_itemize}
 \item Assume first that $X_{ij}$ is small.
   If $w \in \tilde{W}(P')$ and  $|N(w) \cap \tilde{B}(P')_{ij}|>1$, 
remove $j$ from  $L'(w)$.
 \item Next assume that $X_{ij}$ is big. Let
   $X_{ij}=\{a_1-b_1-c_1, \ldots, a_{M}-b_{M}-c_{M}\}$.
   If $b \in \tilde{B}(P',i) \cap \tilde{B}(P)$ and $b$ is adjacent to $b_q$
   for more than one value of $q$, remove $i$
   from $L'(b)$.
 \end{packed_itemize}
 Let $P_{X,c}$ be the $r$-seeded precoloring thus obtained.
 We list several properties of $P_{X,c}$.

 \begin{equation}
  \label{QR}
   \longbox{ 
     Let $i,j \in \{1,2,3\}$.
     \begin{packed_itemize}
       \item $W(P_{X,c})$ is anticomplete to $V(X)$.
   \item If $b \in B(P_{X,c},i)$ then $b$ is anticomplete to $U_{i,j}$.
    \item If $b \in \tilde{B}(P_{X,c},i)$ and $X_{ij}$ is big, then $u$ has
      neighbors in at most one of the seagulls of $X_{ij}$.
   \end{packed_itemize}}
 \end{equation}
The first two statements of \eqref{QR} follow from that fact that
$V(X) \subseteq S(P_{X,c})$. By the second bullet of the construction process of
$P_{X,c}$, we deduce that
 if $X_{ij}$ is big and  $b$ has neighbors in more than  one of the seagulls of
 $X_{ij}$, then $i$ is removed from the list of $b$; thus the third statement of
 \eqref{Q} follows. This proves \eqref{QR}.

\begin{equation}
   \label{smallguessR}
\longbox{ Let $i,j \in \{1,2,3\}$.
  If $X_{ij}$  is small, then no vertex
  $w \in \tilde{W}(P_{X,c})$  has two neighbors in
  $B(P_{X,c})_{ij}$.}
   \end{equation}
Suppose that such $w$ has two  neighbors $n,n'$ in $B(P_{X,c})_{ij}$.
Then $n,n' \in \tilde{B}(P_{X,c})_{ij}$.
Since $P_{X,c}$ is obtained from
 $P$ by moving a set of vertices to the seed, it follows that
 $B(P_{X,c}) \subseteq B(P) \cup W(P)$.
 By \eqref{safe} no vertex of $W(P)$ is adjacent to $w$,
 and therefore $n,n' \in \tilde{B}(P)$.
 It follows that  in the construction process of $L(P_{X,c})$, $j$ was
 removed from the list of $w$, contrary to the fact that  $w \in \tilde{W}(P_{X,c})$.
  This proves~\eqref{smallguessR}.

 \bigskip
 
 Let $\mathcal{L}_R$ be the collection of all precolorings $P_{X,c}$ as above
 where  $X \in \mathcal{Q}$ and $c$ is a consistent precoloring of $V(X)$.

   \begin{equation}
     \label{LRsize}
     \longbox{\begin{packed_itemize}
       \item $|\mathcal{L}_R| \leq (3|V(G)|)^{6 \max(f_R(M),3M)}$.
       \item  $|S(Q)| \leq |S(P)|+6 \max(f_R(M),3M)$ for every $Q \in \mathcal{L}_R$.
       \end{packed_itemize}}
   \end{equation}
   For every $X \in \mathcal{Q}$, $V(X)=\bigcup_{i \neq j \in \{1,2,3\}}V(X_{ij})$,
   and $|V_{ij}| \leq \max(f_R(M),3M)$. Thus
   $|V(X)| \leq 6 \max(f_R(M),3M)$. It follows that
   $|S(Q)| \leq |S(P)|+6 \max(f_R(M),3M)$ for every $Q \in \mathcal{L}_R$.
   Moreover,  there are at most
   $|V(G)|^{ 6 \max(f_R(M),3M)}$ possible choices for $X$. Since there are at most
   $3^{|X|} \leq 3^{6 \max(f_R(M),3M)}$ possible colorings of $G|X$, it follows that
   $|\mathcal{L}_R| \leq (3|V(G)|)^{6 \max(f_R(M),3M)}$. This proves~\eqref{LRsize}.

 \bigskip

Let $Q \in \mathcal{L}_R$ and   $\{i,j,k\}=  \{1,2,3\}$.
Let $f_S$ be as in \ref{SLemma} and let  $N=\max(2^{|S(Q)|+1}, r+6)$.
We enumerate all $S,N$-characteristics of $Q$.
 Let $smallguess(W,k)$ be the set of all subsets of
 $\tilde{B}(Q,k)$ of size at most $f_S(N)$, and $bigguess(Q,k)$ be the set
  of all flocks of size $N$ such that every seagull of the flock
  has one wing in $\tilde{B}(Q)_{ik}$, body $w \in \tilde{W}(Q)$, and the
  other wing in $\tilde{B}(Q)_{jk}$.
    Let $guess(Q,k)=smallguess(Q,k) \cup bigguess(Q,k)$.
We say that $X_k \in guess(Q,k)$ is {\em small} if
  $X_k  \in smallguess(Q,k)$ and that $X_k$ is {\em  big} if
  $X_k  \in bigguess(Q,k)$. If $X_k$ is big, we denote
by $U_k$ the set of the wings of the flock, by $W_k$ the set of the bodies
of the flock, and write
  $V_k=U_k\cup W_k$.  If $X_k$ is small, we write
  $V_k=U_k=X_k$, and $W_k=\emptyset$.  In both cases $V(X_k)=V_k$.
 A precoloring $c$ of $(G|V_k, L(Q))$ is {\em $k$-consistent} if
  $c(v)=k$ for every $v \in U_k$.

 Let $\mathcal{S}(Q)$ be the set of all triples
 $X=(X_1,X_2,X_3)$ such that $X_k \in guess(Q,k)$. Let
 $V(X)=V(X_1) \cup V(X_2) \cup V(X_3)$. We say that a precoloring $c$ of $(G|V(X), L(Q))$
 is {\em consistent} if
 $c|V_i$ is $i$-consistent for all $i$.

 Let  $X \in \mathcal{S}(Q)$ and let $c$ be a  consistent precoloring of
 $(G|V(X),L(Q))$.
 Denote by $P'=(G,L',S')$  the precoloring obtained from $Q$ by moving $V(X)$ to the
 seed with $c$. We modify $L'$ further. For every 
 $i, j \in \{1,2,3\}$ with $i \neq j$, proceed as follows.

 \begin{packed_itemize}
 \item Assume first that both $X_i$  and $X_j$ are small.
   If $w \in \tilde{W}(P')$ has a neighbor in all three of the sets
   $\tilde{B}(P')_{ij}, \tilde{B}(P')_{ik}, \tilde{B}(P')_{jk}$, remove $k$ from $L'(w)$.
 \item Next assume that $X_i$ is big. Let
   $X_i=\{a_1-b_1-c_1, \ldots, a_{N}-b_{N}-c_{N}\}$.
      If $b \in \tilde{B}(P',i) \cap \tilde{B}(Q)$ and $b$ is adjacent to $b_q$
   for more than one value of $q$, remove $i$
   from $L'(b)$.
\item Next assume that $X_j$ is big. Let
  $X_j=\{a_1-b_1-c_1, \ldots, a_{N}-b_{N}-c_{N}\}$.
     If $b \in \tilde{B}(P',j) \cap \tilde{B}(Q)$ and $b$ is adjacent to $b_q$
   for more than one value of $q$, remove $j$
   from $L'(b)$.
 \end{packed_itemize}
 Let $Q_{X,c}$ be the $r$-seeded precoloring thus obtained.
 We list several properties of $Q_{X,c}$.
\begin{equation}
  \label{QS}
   \longbox{ 
     Let $i,j \in \{1,2,3\}$.
     \begin{packed_itemize}
       \item $W(Q_{X,c})$ is anticomplete to $V(X)$.
   \item If $b \in B(Q_{X,c},i)$, then $b$ is anticomplete to $U_i$.
    \item If $b \in \tilde{B}(Q_{X,c},i)$ and $X_i$ is big, then $u$ has
      neighbors in at most one of the seagulls of $X_i$.
  \end{packed_itemize}}
 \end{equation}
The first two statements of \eqref{QS} follow from that fact that
$V(X) \subseteq S(Q_{X,c})$. We now prove the third bullet.
Let $b \in \tilde{B}(Q_{X,c},i)$ and suppose that $b$ has neighbors in more
that one of the seagulls of $X_i$.
Since $Q_{X,c}$ is obtained from $Q$ by moving vertices to the seed,
it follows that $B(Q_{X,c}) \subseteq B(Q) \cup W(Q)$, and 
thus $b \in B(Q) \cup W(Q)$. Moreover, $N(b) \cap V(X_i) \subseteq W_i$.
Since $W_i \subseteq \tilde{W}(Q)$, it follows that no two vertices of $W_i$
belong to the same component of
$W(Q)$, and thus $b \not \in W(Q)$.
Consequently, $b \in \tilde{B}(Q)$.
By the second bullet of the construction process of
$Q_{X,c}$, we deduce that
 if $b$ has neighbors in more than  one of the seagulls of
 $X_i$, then $i$ is removed from the list of $b$; thus the third statement of
 \eqref{QS} follows. This proves \eqref{QS}.

\begin{equation}
   \label{doublesmallguessS}
\longbox{ Let $i,j \in \{1,2,3\}$ be distinct.
  If $X_i$  and $X_j$ are both  small, then $|L(Q_{X,c})(w)|<3$
  for every   $w \in W(Q_{X,c})$ with a neighbor in
all three of the sets $B(Q_{X,c})_{ij},B(Q_{X,c})_{ik}$,
  $B(Q_{X,c})_{jk}$.}
     \end{equation}
\eqref{doublesmallguessS} follows immediately from the first
bullet of the description of the modification of $L'$.

\begin{equation}
   \label{onesmallguessS}
\longbox{ Let $i \in \{1,2,3\}$.
  If $X_i$  is big, then $G|(\tilde{B}(Q_{X,c},i) \cup \tilde{W}(Q_{X,c}))$ is $P_6$-free.}
     \end{equation}
Suppose that $R$ is a $P_6$ in $G|(\tilde{B}(Q_{X,c},i) \cup \tilde{W}(Q_{X,c}))$.
By \eqref{QS} every vertex of $R$ has neighbors in at most one seagull
of $X_i$. Consequently at least $N-6 \geq r$ seagulls of $X_i$
are anticomplete to $R$, contrary to the the fact that $G$ is
$(P_6+rP_3)$-free. This proves~\eqref{onesmallguessS}.

\begin{equation}
\label{bigguessS}
\longbox{Write $Q=P_{Y,d}$ and let $\{i,j,k\}=\{1,2,3\}$.
  If $X_i$ is big, then $Y_{ij}$ and $Y_{ik}$ are small.}
\end{equation}
Suppose that \eqref{bigguessS} is false; by symmetry we may assume that
$Y_{ij}$ is big. Since $X_i$ is big, there is a type $T$ of $S(Q)$
with $L(T)=j$ such that at least ${N \over {2^{|S(Q)|}}} \geq 2$ of the seagulls of
$X_i$ have a wing in $\tilde{B}(Q,T)$. Let $a_1-b_1-c_1$
and $a_2-b_2-c_2$ be such seagulls, where $a_1,a_2 \in \tilde{B}(Q,T)$.
 Let $s \in T$. Then 
$R=b_1-a_1-s-a_2-b_2-c_2$ is a $P_6$, and by \eqref{QR}
every vertex of $R$ has neighbors in at most one of the seagulls of
$Y_{ij}$. Therefore $V(R)$ is anticomplete to at least
$M-6 \geq r$ of the seagulls of $Y_{ij}$, contrary to the fact that
$G$ is $(P_6+rP_3)$-free. This proves \eqref{bigguessS}.

\bigskip

Let $\mathcal{L}_S(Q)$ be the list of all precolorings $Q_{X,c}$ where
$X \in \mathcal{S}(Q)$ and $c$ is a consistent precoloring of
$(G|V(X),L(Q))$. We claim the following:

   \begin{equation}
     \label{LS(Q)size}
     \longbox{
       \begin{packed_itemize}
           \item    $|\mathcal{L}_S(Q)| \leq (3|V(G)|)^{3 \max(f_S(N),3N)}$.
           \item $S(S) \leq |S(Q)|+3 \max(f_S(N),3N)$ for every $S \in \mathcal{L}_S(Q).$
      \end{packed_itemize}       }
   \end{equation}
   For every $X \in \mathcal{S}$, $V(X)=V_1 \cup V_2 \cup V_3$,
   and $|V_i| \leq \max(f_S(N),3N)$. Thus
   $|V(X)| \leq 3 \max(f_S(N),3N)$.
It follows that $S(Z) \leq |S(Q)|+3 \max(f_S(N),3N)$ for every $Z \in \mathcal{L}_S(Q).$
   Moreover, there are at most
   $|V(G)|^{ 3 \max(f_S(N),3N)}$ possible choices for $X$. Since there are at most
   $3^{|X|} \leq 3^{3 \max(f_S(N),3N)}$ possible colorings of $G|V(X)$, it follows
   that
   $|\mathcal{L}_S(Q)| \leq (3|V(G)|)^{3 \max(f_R(N),3N)}$. This
   proves~\eqref{LS(Q)size}.

   \bigskip

 Let $S \in \mathcal{L}_S(Q)$. 
 Let $T=r+6$ and let  $\{i,j,k\}=\{1,2,3\}$.
 Our next step is to enumerate all heterogeneous $T$-characteristics of $S$.
 A {\em potential $(S,i,j)$-key} is a pair $(X_1,X_2)$ such that
 \begin{packed_itemize}
\item $X_1$ is a flock of $ij$-typed seagulls of $S$ with $|X_1| < T$.
Write $X_1=\{x_1-y_1-z_1, \ldots, x_m-y_m-z_m\}$
  where $x_1, \ldots, x_m \in \tilde{B}(S)_{ik}$,
  $y_1, \ldots, y_m \in \tilde{W}(S)$, and
  $z_1, \ldots, z_m \in \tilde{B}(S)_{jk}$.
\item
 Let $c$ be a coloring of $G|V(X_1)$ such that for every
  $l \in \{1, \ldots, m\}$,
  $c(x_l)=i$, $c(z_l)=j$ and $c(y_l)=k$.
  Let $P'=(G,L',S')$ be obtained from $S$ by moving $V(X_1)$ to the seed
  with $c$.
  For every $l \in \{1, \ldots, m\}$
  let $S_l$ be a  flock of size at most one, such that
  if $S_l \neq \emptyset$, then the member of $S_l$ is an $ij$-typed seagull
  of $P'$ related to $x_l-y_l-z_l$. Let $X_2=\bigcup_{l=1}^m S_l$. 
\end{packed_itemize}

 Let $smallguess(S,i,j)$ be the set of all potential $(S,i,j)$-keys
and $bigguess(S,i,j)$ be the set
  of all flocks of size $T$ such that every seagull of the flock
is an $ij$-typed seagull of $S$.
    Let $guess(S,i,j)=smallguess(S,i,j) \cup bigguess(S,i,j)$.
We say that $X_{ij} \in guess(S,i,j)$ is {\em small} if
  $X_{ij}  \in smallguess(S,i,j)$ and that $X_{ij}$ is {\em  big} if
  $X_{ij}  \in bigguess(S,i,j)$. If $X_{ij}$ is big, we denote
  by $U_{ij}^i$ the set of the wings of the flock that are contained in
  $\tilde{B}(S)_{ik}$, by $U_{ij}^j$ the set of the wings of the flock that are
  contained in $\tilde{B}(S)_{jk}$, and 
  by $W_{ij}$ the set of the bodies of the flock, and write
  $V_{ij}=U_{ij}^i \cup U_{ij}^j  \cup W_{ij}$.  Next assume that
  $X_{ij}=(X_1^{ij},X_2^{ij})$ is small.
Denote by $U_{ij}^i$ the set of the wings of $X_1^{ij}$ that are contained in
  $\tilde{B}(S)_{ik}$, by $U_{ij}^j$ the set of the wings of $X_1^{ij}$ that are
  contained in $\tilde{B}(S)_{jk}$, and 
  by $W_{ij}$ the set of the bodies of $X_1^{ij}$.
  We write $V_{ij}=U_{ij}^i \cup U_{ij}^j \cup W_{ij} \cup V(X_2^{ij})$.
  In both cases $V(X_{ij})=V_{ij}$.
  A precoloring $c$ of $(G|V_{ij}, L(S))$ is {\em $ij$-consistent} if
  \begin{packed_itemize}
  \item  $c(v)=i$ for every $v \in U_{ij}^i$,
  $c(v)=j$ for every $v \in U_{ij}^j$,
   $c(v)=k$ for every $v \in W_{ij}$, and
    \item if $X_{ij}$ is small, then every seagull of $X_2^{ij}$ has at least
      one  wing $v$ with $c(v)=k$.
  \end{packed_itemize}
  Let $\mathcal{T}(S)$ be the set of all triples $X=(X_{12},X_{13},X_{23})$
  such that $X_{ij} \in guess(S,i,j)$. Let
  $V(X)=V(X_{12}) \cup V(X_{13}) \cup V(X_{23})$.
  We say that a precoloring $c$ of $(G|V(X), L)$
 is {\em consistent} if
 $c|V_{ij}$ is $ij$-consistent for all $i<j$.

 Let $X \in \mathcal{T}(S)$ and let $c$ be a  consistent precoloring of
 $G|V(X)$.
 Let $P'=(G,L',S')$ be the precoloring obtained from $S$ by moving $V(X)$ to the
 seed with $c$. If $X_{ij}$ is small, we
 modify $L'$ further, as follows. Let $i,j \in \{1,2,3\}$.
 Write $X_1^{ij}=\{x_1-y_1-z_1,\ldots, x_m-y_m-z_m\}$ and
 $X_2^{ij}=\{S_1^{ij}, \ldots, S_m^{ij}\}$ as in the definition of a
 potential key.
  \begin{packed_itemize}
 \item    Let $l \in \{1, \ldots, m\}$. If $S_l^{ij}=\emptyset$,
   and $p-q-r$ is an $ij$-typed seagull of $P'$ related to $x_l-y_l-z_l$,
   remove $k$ from $L'(p)$ and from $L'(r)$.
 \item If $p-q-r$ is an $ij$-typed seagull of $P'$ and $X_1^{ij} \cup \{p-q-r\}$ is
   a flock, remove $k$ form $L'(q)$.
  \end{packed_itemize}
 Let $S_{X,c}$ be the $r$-seeded precoloring thus obtained.
We list several properties of $S_{X,c}$.
\begin{equation}
  \label{QT}
   \longbox{ 
     Let $\{i,j,k\} \in \{1,2,3\}$.
     \begin{packed_itemize}
       \item $W(S_{X,c})$ is anticomplete to $V(X)$.
       \item If $b \in B(S_{X,c},i)$ then $b$ is anticomplete to every
         $v \in V(X)$ with $c(v)=i$.
       \item $B(S_{X,c},k)$ is anticomplete to  $W_{ij}$.
  \end{packed_itemize}}
 \end{equation}
All three statements of \eqref{QT} follow from that fact that
$V(X) \subseteq S(S_{X,c})$.

\begin{equation}
   \label{smallguessT}
   \longbox{ Let $\{i,j,k\} =\{1,2,3\}$ and assume that
     $X_{ij}$ is small.
     Write $X_1^{ij}=\{x_1-y_1-z_1,\ldots, x_m-y_m-z_m\}$ as in the definition of
     a  potential key.
  \begin{packed_enum}
  \item    Let $l \in \{1, \ldots, m\}$. If $S_l^{ij}=\emptyset$,
    then no $ij$-typed seagull of $S_{X,c}$ is related to $x_l-y_l-z_l$.
  \item  There is no  $ij$-typed seagull $p-q-r$  of $S_{X,c}$ such that
    $X_1^{ij} \cup \{s\}$ is a flock.
  \end{packed_enum}}
     \end{equation}
\eqref{smallguessT} follows immediately from the the description
of the modification of $L'$, and the fact that, by
\ref{convention}, an $ij$-typed seagull
as in \eqref{smallguessT}.2  is also an $ij$-typed seagull of $P'$ (with the
notation as in the description of the modification of $L'$).

\bigskip

Recall that there exists an $r$-seeded precoloring $Q$ such that
  $S \in \mathcal{L}_S(Q)$
and $Q \in \mathcal{L}_R(P)$. Let $Z$ and $d$ be such that
$Z$ and $d$ be such that   $S=Q_{Z,d}$.
\begin{equation}
   \label{smallguessTgood}
   \longbox{ Let $\{i,j,k\} =\{1,2,3\}$ and assume that
     $X_{ij}$ is small and $Z_k$ is big. Then there is no $ij$-typed seagull in
     $S_{X,c}$.}
\end{equation}
Since $Z_k$ is big, \eqref{onesmallguessS} implies that
$G|(\tilde{B}(S,k) \cup \tilde{W}(S))$ is
  $P_6$-free.
Suppose $p-q-r$ is $ij$-typed seagull of $S_{X,c}$.
Then  $p-q-r$ is an $ij$-typed seagull of $P'$ (with the
notation as in the description of the modification of $L'$).
Then $k \in L(S_{X,c})(p) \cap L(S_{X,c})(r)$.
It follows from ~\eqref{smallguessT}.2 that
$X_1^{ij} \cup \{p-q-r\}$ is not a flock, and so by
\ref{keylemma} $p-q-r$ is related to some seagull of $X_1^{ij}$,
say to $a_1-b_1-c_1$. 
We claim that $S_1=\emptyset$.
Suppose not; write $S_1=\{s_1\}$. Then, by \eqref{QT} and  \ref{keylemma},
$p-q-r$ is related to $s_1$. But
$V(s_1) \subseteq S(S_{X,c})$, and at least one wing
of $s_1$ has color $k$, a contradiction.
This proves that $S_1=\emptyset$, and 
we get a contradiction to  \eqref{smallguessT}.1
This proves~\eqref{smallguessTgood}.

\bigskip

Let $\mathcal{L}_T(S)$ be the list of all precolorings $S_{X,c}$ where
   $X \in \mathcal{T}(S)$ and $c$ is a consistent precoloring of $V(X)$.

   \begin{equation}
     \label{LT(S)size}
\longbox{
\begin{packed_itemize}
\item $ |\mathcal{L}_T(S)| \leq (3|V(G)|)^{18T}.$
\item $|S(T)| \leq |S(S)|+18T$ for every $T \in \mathcal{L}_T(S)$.
  \end{packed_itemize}}
   \end{equation}
   For every $X \in \mathcal{T}$, $X$ consists of three parts, each of which
   is a set of at most $2T$ seagulls.
   Therefore
   $|V(X)| \leq 18T$. It follows that $|S(T)| \leq |S(S)|+18T$ for every $T \in \mathcal{L}_T(S)$. Moreover, 
   there are at most
   $|V(G)|^{ 18T}$ possible choices for $X$. Since there are at most
   $3^{|X|} \leq 3^{18T}$ possible colorings of $G|V(X)$, it follows that
   $|\mathcal{L}_T(S)| \leq (3|V(G)|)^{18T}$. This
   proves~\eqref{LT(S)size}.

   \bigskip

   Let $R \in \mathcal{L}_T(S)$ and let  $\{i,j,k\}=\{1,2,3\}$.
Let $B=r+2$.
   Let $f$ be as in \ref{SLemmatypes}.
   Finally we enumerate all possible $B$-characteristics of pairs of
   types of $R$.
 For every pair of types $T_i,T_j$ of $S(R)$ with $L(R)(T_i)=i$ and
   $L(R)(T_j)=j$, we define the following sets.
 Assume first that  $Z_k$ is small (here $Z_k$ is in the same notation as
 in \eqref{smallguessTgood}). Let
  $smallguess(T_i,T_j)=\{\emptyset\}$ and $bigguess(T_i,T_j)=\emptyset$.
Next assume that $Z_k$ is big.
   Let  $smallguess(T_i,T_j)$ be the set of all subsets of
   $\tilde{B}(R,T_i) \cup \tilde{B}(R,T_j)$ of size at most $f(B)$, and $bigguess(T_i,T_j)$ be
   the set
  of all flocks of size $B$ such that every seagull of the flock
  has a wing in $\tilde{B}(R,T_i)$, a wing in $\tilde{B}(R,T_j)$ and body in
  $w \in \tilde{W}(R)$.

  In all cases, let
  $guess(T_i,T_j)=smallguess(T_i,T_j) \cup bigguess(T_i,T_j)$.
  Let $\mathcal{T}$ be the set of all types of $S(R)$, say
  $|\mathcal{T}|=t$.
  Now let $\mathcal{C}(R)$ be the set of all vectors
  $(X_{T_i,T_j})$ where $T_i,T_j  \in \mathcal{T}$
with $L(R)(T_i)=i$ and
   $L(R)(T_j)=j$, and
$X_{T_i,T_j} \in guess(T_i,T_j)$. Then $X$ has at most $t^2$ components.

We say that $X_{T_i,T_j}$ is {\em small} if
  $X_{T_i,T_j} \in smallguess(T_i,T_j)$ and that $X_{T_i,T_j}$ is {\em  big} if
  $X_{T_i,T_j} \in bigguess(T_i,T_j)$. If $X_{T_i,T_j}$ is big, we denote
by $U_{T_i,T_j}$ the set of the wings of the flock, and   by $W_{T_i,T_j}$ the
set of the bodies of the flock,
  and write   $V_{T_i,T_j}=U_{T_i,T_j} \cup W_{T_i,T_j}$.
  If $X_{T_i,T_j}$ is small, we write
  $V_{T_i,T_j}=U_{T_i,T_j}=X_{T_i,T_j}$ and $W_{T_i,T_j}=\emptyset$.
  Finally, let
  $V(X)=\bigcup_{T_i,T_j}V_{T_i,T_j}$.

 A precoloring $c$ of $G|V(X)$ is {\em consistent} if
  $c(v)=k$ for every $v \in U_{T_i,T_j}$.
  Let $X \in \mathcal{C}(R)$ and let $c$ be a  consistent precoloring of
  $G|V(X)$.
  We construct the $r$-seeded precoloring $R_{X,c}$ as
  follows.
  Let $P'=(G,L',S')$ be obtained from $R$ by moving $V(X)$ to the seed with
  $c$.   Next we modify $L'$ further.
Please note that we are still dealing with types of
 $R$, and not with types of $P'$. 
 \begin{packed_itemize}
 \item Let $\{i,j,k\}=\{1,2,3\}$,  let $T_i,T_j,T_k$ be such that
   $L(T_l)=l$, and and assume that
   $X_{T_i,T_j}$ and $X_{T_i,T_k}$ are both small.
   If $w \in \tilde{W}(P')$ is both in a seagull with wings in $\tilde{B}(R,T_i)$ and $\tilde{B}(R,T_j)$,
   and in a seagull with wings in $\tilde{B}(R,T_i)$ and $\tilde{B}(R,T_k)$,
   remove $i$ from $L'(w)$.
 \item Let $\{i,j,k\}=\{1,2,3\}$,  let $T_i,T_j,T_k$ be such that $L(T_l)=l$,
   and assume that
   $X_{T_i,T_j}$ is small. If $w \in \tilde{W}(P')$ is in a seagull with wings in $\tilde{B}(R,T_i)$ and $\tilde{B}(R,T_j)$, and 
   $N(w) \cap \tilde{B}(R,T_i)$ is complete to $N(w) \cap \tilde{B}(R,T_k)$,
   then   remove $i$ from $L'(w)$.
  \end{packed_itemize}
Denote the precoloring thus obtained by $R_{X,c}$.

 \begin{equation}
   \label{badvertices}
\longbox{   Let $w \in \tilde{W}(R_{X,c})$. Then either
\begin{packed_enum}      
\item  $w$ only has neighbors in at most two of $B(R_{X,c})_{ij}$,
  or
\item   there exist $T_1,T_2,T_2',T_3$ such that 
   $L(R)(T_1),L(R)(T_2),L(R)(T_3)$ are all distinct,
   $L(R)(T_2)=L(R)(T_2')$,
      and $w$ is both in a seagull with wings in $\tilde{B}(R_{X,c},T_1)$ and
   $\tilde{B}(R_{X,c},T_2)$,
      and in a seagull with wings in $\tilde{B}(R_{X,c},T_2')$ and $\tilde{B}(R_{X,c},T_3)$.
      Moreover, if $T_2$ and $T_2'$ can be chosen with $T_2=T_2'$, then
      $X(T_1,T_2)$ and $X(T_2,T_3)$ are both big.
\end{packed_enum}}
 \end{equation}
 Let $\{i,j,k\}=\{1,2,3\}$ and
 let $N_{k}=N(w) \cap B(R_{X,c})_{ij}$. We may assume that all three of the sets
 $N_1,N_2,N_3$ are non-empty, for otherwise \eqref{badvertices}.1 holds.
 Since $w \in \tilde{W}(R_{X,c})$, it follows that  $w \in \tilde{W}(P)$.
  By \eqref{safe} $N(w) \subseteq \tilde{B}(P)$.
 Since $P$ is stable, we deduce that $N(w)$ is not connected.
 Consequently, $N_1$ is not complete to $N_2 \cup N_3$,
 $N_2$ is not complete to $N_1 \cup N_3$, and $N_3$ is not complete to
 $N_1 \cup N_2$.
If follows that  there exist types $T_1,T_2,T_2',T_3$ of $R$ such that 
   $L(R)(T_1),L(R)(T_2),L(R)(T_3)$ are all distinct,
   $L(R)(T_2)=L(R)(T_2')$,
   and $w$ is both in a seagull with wings in $\tilde{B}(R_{X,c},T_1)$ and
   $\tilde{B}(R_{X,c},T_2)$,
  and in a seagull with wings in $\tilde{B}(R_{X,c},T_2')$ and $\tilde{B}(R_{X,c},T_3)$.
Now suppose  $T_2$ and $T_2'$ can be chosen with $T_2=T_2'$.
  By the first bullet of the construction process of $L(R_{X,c})$ we may assume
  that $X(T_1,T_2)$ is big.  We may also assume that
  $X(T_2,T_3)$ is small, for otherwise \eqref{badvertices}.2 holds.
  Now by the second bullet point of the construction of $L(R_{X,c})$,
  we deduce that $w$ is also in a seagull with wings in
  $\tilde{B}(R_{X,c},T_1)$ and $\tilde{B}(R_{X,c},T_3)$. But now
  $X(T_1,T_3)$ is big (by the first  bullet point of the construction of
  $L(R_{X,c})$), and again  \eqref{badvertices}.2 holds.
    This proves \eqref{badvertices}.

 \bigskip
 
  Write $Q=P_{Y,d}$, $S=Q_{Z,f}$ and $R=S_{U,g}$
  \begin{equation}
    \label{smalltypes}
    \longbox{
  $|L(R_{X,c})(w)|<3$ for every $w \in W(R_{X,c})$ 
    with neighbors in all three 
    of the sets $B(R_{X,c})_{ij},B(R_{X,c})_{ik}, B(R_{X,c})_{jk}$
    (here $\{i,j,k\}=\{1,2,3\}$).}
  \end{equation}
  We may assume that there exists $w \in \tilde{W}(R_{X,c})$ with 
     neighbors in all three 
      of the sets $B(R_{X,c})_{ij}$, $B(R_{X,c})_{ik}$, $B(R_{X,c})_{jk}$.
  By \eqref{doublesmallguessS} there exist
  distinct $i,j \in \{1,2,3\}$ such that both $Z_i$ and $Z_j$ are big.
  It follows from \eqref{bigguessS}
  that $Y_{12}$, $Y_{13}$ and $Y_{23}$ are all small,
  and so every $w' \in \tilde{W}(Q)$ 
 has at most one neighbor in $B(Q)_{ij}$ for every
 $i,j \in \{1,2,3\}$. Using \eqref{safe} we deduce  that
  every $w' \in \tilde{W}(R_{X,c})$ 
 has at most one neighbor in $B(R_{X,c})_{ij}$ for every
 $i,j \in \{1,2,3\}$. It now follows from \eqref{safe} that
$deg(w)>2$, and so
 \eqref{badvertices}.2 holds for $w$ and $T_2=T_2'$. With the notation
 of \eqref{badvertices}.2, 
 we may assume that $L(T_i)=i$; consequently $w$ is both in
 a $23$-seagull, and in a $12$-seagull. 
Using symmetry, we may assume that $Z_1$ is big.
 Now by \eqref{smallguessTgood}, it follows that $U_{23}$
 is big. Let $U_{23}' \subseteq U_{23}$ be a flock of size exactly $r$.
 Let $X_{T_1,T_2}=\{p_1-q_1-r_1, \ldots, p_B-q_B-r_B\}$, where
 $p_1, \ldots, p_B \in \tilde{B}(R,T_1)$
 and $r_1, \ldots, r_B \in \tilde{B}(R,T_2)$ (recall that $X_{T_1,T_2}$ is big
 by \eqref{badvertices}). Since each of the bodies of the seagulls
 of $U_{23}'$ has at most one neighbor in
 $\tilde{B}(Q)_{23}$, it follows that the set of   bodies of $U_{23}'$
 is anticomplete to at least $B-r=2$ of $p_1, \ldots, p_B$.
 We may assume that the set of bodies of $U_{23}'$ is anticomplete
 to $\{p_1,p_2\}$.
 Let $s \in T_1$. Then
 $M=q_2-p_2-s-p_1-q_1-r_1$  and $M'=q_1-p_1-s-p_2-q_2-r_2$ are copies of $P_6$
 in $G$.
If every vertex of $M$ has neighbors in at most one seagull
of $Z_1$, then $V(M)$ is anticomplete to at least
$N-6 \geq r$ seagulls of $Z_1$, contrary to the fact that
$G$ is $(P_6+rP_3)$-free. By \eqref{QS} we may assume
that $p_1$ has neighbors in at least two seagulls of $Z_1$,
say $a_1-b_1-c_1$ and $a_2-b_2-c_2$. 
It follows that $q_1,b_1,b_2 \in \tilde{W}(P)$. Now  \eqref{safe}
implies that $p_1,a_1,c_1,a_2,c_2 \in \tilde{B}(P)$, and
since $P$ is stable, 
$p_1$ is not complete to either of $\{a_1,c_1\}$, $\{a_2,c_2\}$.
Also, since $L(R_{X,c})(a_1)=L(R_{X,c})(c_1)=L(R_{X,c})(a_2)=L(R_{X,c})(c_2)=\{1\}$, it follows
that $L(R_{X,c})(b_1) \neq 1$ and $L(R_{X,c})(b_2) \neq \{1\}$,
and since $p_1 \in \tilde{B}(R_{X,c})_{23}$, we have that $p_1$
is anticomplete to $\{b_1,b_2\}$.
We deduce that $G|\{p_1,a_1,b_1,c_1,a_2,b_2,c_2\}$ contains a $P_6$, say $K$.
Recall that each of $b_1,b_2$ has at most one neighbor in each of
$B(Q)_{12}$ and $B(Q)_{13}$ (namely $a_1,c_1$ and $a_2,c_2$, respectively). Since
$V(Z_1) \subseteq S(Q_{Z,f})$, it follows that $\{a_1,b_1,c_1, a_2,b_2,c_2\}$
is anticomplete to $V(U_{23})$.
Now, since $U_{23}'$ are all $23$-typed seagulls of $S=Q_{Z,f}$ with bodies
anticomplete to $p_1$,
and since $p_1 \in \tilde{B}(R_{X,c})_{23} \subseteq \tilde{B}(S)_{23} \cup \tilde{W}(S)$, it follows that
$\{p_1,a_1,b_1,c_1,a_2,b_2,c_2\}$ is anticomplete to $V(U_{23}')$.
Since $U_{23}'$ is a flock of size  $r$, this contradicts the
fact that $G$ is $(P_6+rP_3)$-free. 
This proves~\eqref{smalltypes}.

\begin{equation}
  \label{smalltypes2}
  \longbox{
  If $w \in \tilde{W}(R_{X,c})$  has neighbors in
  $B(R_{X,c})_{ij}$ and in $B(R_{X,c})_{ik}$, then for every pair of types
  $T_1, T_2$ with $L(R)(T_1)=k$ and $L(R)(T_2)=j$, we have that 
  $X(T_1,T_2)$ is small.}
\end{equation}
By \eqref{smalltypes} we may assume that $w$ is anticomplete to
$\tilde{B}(R_{X,c})_{jk}$.
  Suppose that there exist $T_1,T_2$ as above with $X(T_1,T_2)$ big.
It follows that $Z_i$ is big. But now by \eqref{bigguessS}
both $Y_{ij}$ and $Y_{ik}$ are small, and so by \eqref{smallguessR}
$deg(w)<2$. By \eqref{safe} we get a contradiction to the fact that
$P$ is stable. This proves~\eqref{smalltypes2}.

\bigskip

Now we construct an augmented $r$-seeded precoloring $M_{X,c}$ as follows.
If $|L(R_{X,c})(w)|=2$ for every $w \in W(R_{X,c})$,
let $M_{X,c}=(R_{X,c},\emptyset)$. Now we may assume that
$\tilde{W}(R_{X,c}) \neq \emptyset$.
Let $W_l$ be the set of $w \in \tilde{W}(R_{X,c})$
with neighbors in exactly $l$ of the sets  $B(R_{X,c})_{ij}$.
If $W_3 \neq \emptyset$, set
$L'(v)=\emptyset$ for every $v \in V(G) \setminus S(R_{X,c})$,
and let $M_{X,c}=((G,L', S(R_{X,c})), \emptyset)$.
Now assume that $W_3=\emptyset$. Let
$G'=G \setminus (W_1 \cup W_2)$, and let $\mathcal{X}$ be the set
of all the non-empty sets $N(w) \cap B(R_{X,c})_{ij}$ with $w \in W_2$.
Let $M_{X,c}=((G',L(R_{X,c}), S(R_{X,c}), \mathcal{X})$     .

Let $\mathcal{M}(R)$ be the set of all the augmented $r$-seeded
precolorings $M_{X,c}$ where $X \in \mathcal{C}$ and $c$ is a
consistent precoloring of $X$.

\begin{equation}
  \label{M(R)size}
\longbox{$|\mathcal{M}(R)|$ is polynomial.}
\end{equation}
The  number of possible  pairs $(T_i,T_j)$ where
$T_i,T_j \in \mathcal{T}$ and $L(R)(T_i) \neq L(R)(T_j)$
is at most  $t^2 \leq 2^{2|S(R)|}$. By
\eqref{LRsize}, \eqref{LS(Q)size}, \eqref{LT(S)size},
there is a constant $D$
that depends on $r$ but not on $G$, such that
$|S(R)| \leq D$.  Since $|V_{T_i,T_j}| \leq max (3B,f(B))$
for every $T_i,T_j$, it follows that for every 
$X \in \mathcal{C}$, $|V(X)| \leq 2^{2D} max (3B,f(B))$.
Let $K= 2^{2D} max (3B,f(B))$. Then there are at most
$|V(G)|^K$ choices for the members of $\mathcal{C}$,
and so
 $|\mathcal{C}| \leq |V(G)|^K$. Moreover, for every
 $X \in \mathcal{C}$, the number of colorings of $G|V(X)$ is at most
$3^{|V(X)|} \leq 3^K$.
 Since $|\mathcal{M}(R)|$ is at most the total number of pairs
 $(X,c)$  where  $X \in \mathcal{C}$ and $c$ is a coloring
 of $G|V(X)$, we deduce that
 $|\mathcal{M}(R)| \leq (3|V(G)|)^K$.
 This proves~\eqref{M(R)size}.

\bigskip

Finally, let
$\mathcal{L}=\bigcup_{Q \in \mathcal{L}_R} \bigcup_{S \in \mathcal{L}_S(Q)}\bigcup_{R \in \mathcal{L}_T(S)}\mathcal{M}(R).$

It follows from \eqref{LRsize}, \eqref{LS(Q)size}, \eqref{LT(S)size} and
\eqref{M(R)size} that $|\mathcal{L}|$ is polynomial.
Moreover, for every $(M, \mathcal{X}) \in \mathcal{L}$
$|L(M)(v)| \leq 2$ for every $v \in V(G(M))$, and $|\mathcal{X}| \leq 2|V(G)|$;
thus $(M, \mathcal{X})$ is tractable.
It remains to show that $\mathcal{L}$ is equivalent to $P$.
Suppose $(M,\mathcal{X}) \in \mathcal{L}$ has a precoloring
extension $d$. 
We observe that $M$ is an $r$-seeded precoloring $(G',L',S')$,
where $G'$ is an induced subgraph of $G$, and $L'(v) \subseteq L(v)$
for every $v \in V(G')$. Thus a precoloring extension $c$ of $M$
is also a precoloring extension of $(G',L,S)$. Next we
observe that if $v \in V(G) \setminus V(G')$, then $|L(w)|=3$, and
either
\begin{packed_itemize}
\item there exist $i,j \in \{1,2,3\}$ such that
  $L'(n)=\{i,j\}$ for every $n \in N(v)$, or
\item there exist $X_1, X_2 \in \mathcal{X}$
  such that $N(v)=X_1 \cup X_2$.
\end{packed_itemize}
In both cases  we have $L(v) \setminus c(N(v)) \neq \emptyset$,
and so $c$ can be extended to a precoloring extension of $P$.

Now we  show the converse.
Let $c$ be a precoloring extension of $P$. We will construct
$(M, \mathcal{X}) \in \mathcal{L}$ that has a precoloring extension.
For every $i,j \in \{1,2,3\}$,
let $X_{ij}=char_{R,M}(P,i,j,c)$, and let $X=(X_{12},X_{21},X_{13},X_{31},X_{23},X_{32})$.
Then $c$ is a consistent coloring of $V(X)$, and so 
$Q=Q_{X,c} \in \mathcal{L}_R$.  We claim that $c$ is a precoloring extension of
$Q$. Suppose that $c(v) \not \in L(Q)(v)$ for some $v \in V(G)$.
Let $P'$ be the precoloring obtained from $P$ by moving
$V(X)$ to the seed with $c|V(X)$. Then there exist $i,j \in \{1,2,3\}$ such
that 
either 
\begin{packed_itemize}
\item $X_{ij}$ is small,
  $v \in \tilde{W}(P')$,  $|N(v) \cap B(P')_{ij}|>1$,
and $c(v)=j$, or
\item $X_{ij}$ is big,
  $v\in \tilde{B}(P',i) \cap \tilde{B}(P)$, and $v$ is adjacent
  to the bodies of at least two seagulls of $X_{ij}$,
  and $c(v)=i$.
\end{packed_itemize}
In the former case \ref{RLemma}.1 implies that $v$ has a neighbor in
$B(P')_{ij}$ of color $j$ in $c$, and in the latter case \ref{RLemma}.2
immediately implies that $c(v) \neq i$, in both cases a contradiction. This
proves that $c$  is a precoloring extension of $Q$.

Let $Y_i=char_{S,N}(Q,i,c)$, and let $Y=(Y_1,Y_2,Y_3)$.
Then $c$ is a consistent coloring of $V(Y)$, and so 
$S=Q_{Y,c} \in \mathcal{L}_S(Q)$.  We claim that $c$ is a precoloring extension
of $S$. Suppose that $c(v) \not \in L(S)(v)$ for some $v \in V(G)$.
Let $Q'$ be the precoloring obtained from $Q$ by moving
$V(Y)$ to the seed with $c|V(Y)$. Then there exist $\{i,j,k\}=\{1,2,3\}$
such that
\begin{packed_itemize}
 \item $X_i$  and $X_j$ are small,
   $v \in \tilde{W}(Q')$ has a neighbor in each of the sets
   $\tilde{B}(Q')_{ij}, \tilde{B}(Q')_{ik},\tilde{B}(Q')_{jk}$, and
   $c(v)=k$, or    
 \item $X_i$ is big, 
   $v \in \tilde{B}(Q',i) \cap \tilde{B}(Q)$ and $v$ is adjacent to the body of at least two
   seagulls of $X_i$, and $c(v)=i$, or
 \item $X_j$ is big,
   $v \in \tilde{B}(Q',j) \cap \tilde{B}(Q)$ and $v$ is adjacent to the body of at least two
   seagulls of $X_j$, and $c(v)=j$.
\end{packed_itemize}   

We deal with the first bullet first. In the case of the first bullet
\ref{SLemma}.1 implies that 
\begin{packed_enum}
\item $v$ does not have neighbors $n \in B(Q')_{ij}$ and $n' \in B(Q')_{ik}$
  with $c(n)=c(n')=i$, and
\item $v$ does not have neighbors $m \in B(Q')_{ij}$ and $m' \in B(Q')_{jk}$
  with $c(n)=c(n')=j$.
\end{packed_enum}  
We claim that $v$ has a neighbor 
$n'' \in B(Q')$ with $c(n'')=k$. Suppose not.
Then $c(N(v) \cap B(Q')_{ik})=i$,
and therefore  $c(N(v) \cap B(Q')_{ij})=j$. But  also
$c(N(v) \cap B(Q')_{jk})=j$,
and therefore  $c(N(v) \cap B(Q')_{ij})=i$, a contradiction. This proves that the first bullet above  does not happen. If the case of the
second bullet \ref{SLemma}.2 immediately implies that $c(v) \neq i$,
and in the case of the third bullet \ref{SLemma}.2 immediately implies that $c(v) \neq j$. Thus we get a contradiction in all cases. This proves that
$c$ is a precoloring extension of $S$.

Let $i,j \in \{1,2,3\}$. Let $Z_{ij}=char_{h,T}(S,i,j,c)$.
Let $Z=(Z_{12},Z_{13},Z_{23})$. Then $c$ is a consistent precoloring of
$V(Z)$. Let $R=S_{Z,c}$, then $R \in \mathcal{L}_T(S)$.
We claim that $c$ is a precoloring extension
of $R$. Suppose that $c(v) \not \in L(R)(v)$ for some $v \in V(G)$.
Let $S'$ be the precoloring obtained from $S$ by moving
$V(X)$ to the seed with $c|V(X)$. Then there exists $\{i,j,k\}=\{1,2,3\}$
such that $Z_{ij}$ is small,
$Z_1^{ij}=\{x_1-y_1-z_1,\ldots, x_m-y_m-z_m\}$ 
 and, with the notation of the definition
of an $(S,i,j,c)$-key,  either
  \begin{packed_itemize}
 \item    there exists $l \in \{1, \ldots, m\}$ with $S_l=\emptyset$,
   $v-q-r$ is an $ij$-typed seagull of $S'$ related to $x_l-y_l-z_l$,
   and $c(v)=k$, or
 \item $p-v-r$ is an $ij$-typed seagull of $S'$, $Z_1^{ij} \cup \{p-v-r\}$ is
   a flock, and $c(v)=k$.
  \end{packed_itemize}
  Suppose first that the first case happens. Since $S_l=\emptyset$, it follows
  from the definition of a key that no seagull related to $x_l-y_l-z_l$ has a
  wing colored $k$ in $c$, a contradiction. Thus the second bullet holds,
  and we get a contradiction to the fact that $Z_1^{ij}$ is a maximal flock of $ij$-colored seagulls. This proves that $c$ is a precoloring extension of $R$.

For every pair of types $T_i,T_j$ of $S(R)$ with $L(R)(T_i)=i$ and
$L(R)(T_j)=j$, let  $U(T_i,T_j)=char_B(R,T_1,T_2,c)$.
Note that, by \ref{SLemma}.1, if $Y_k$ is small,
then no vertex of $\tilde{W}(R)$ has neighbors $n \in B(R,T_i)$
and $n' \in B(R,T_j)$ with $c(n)=c(n')=k$, and therefore
$U(T_i,T_j)=\emptyset$.
Let $U=(U(T_i,T_j))$ (so $U$ is a vector indexed by pairs $T_i,T_j$).
Then $c$ is a consistent precoloring of $V(U)$.
Let $D=R_{U,c}$. 
Let $R'$ be the seeded precoloring obtained from $R$ by moving
$V(U)$ to the seed with $c$. We claim that $c$ is a precoloring
extension of $D$. 

Suppose $c(v) \not \in L(D)(v)$ for some
$v \in V(G)$. Then there exist
types $T_i,T_j,T_k$ of $R$  where 
$\{i,j,k\}=\{1,2,3\}$ and $L(R)(T_l)=l$ for every $l \in \{1,2,3\}$, and such
that
  (please note that we are still dealing with types of
 $R$, and not with types of $R'$)
 \begin{packed_itemize}
 \item  $U(T_i,T_j)$ and $U(T_i,T_k)$ are both small,
   $v \in \tilde{W}(R)$ is in both in a seagull with wings in $B(R,T_i)$ and $B(R,T_j)$
and in a seagull with wings in $B(R,T_i)$ and $B(R,T_j)$,
and $c(v)=i$.
\item   $X_{T_i,T_j}$ is small, and
   $N(w) \cap \tilde{B}(R,T_i)$ is complete to $N(w) \cap \tilde{B}(R,T_k)$,
   and  $v \in \tilde{W}(R)$ is in a seagull with wings in $\tilde{B}(R,T_i)$ and $\tilde{B}(R,T_j)$, and $c(v)=i$.
  \end{packed_itemize}
Now \ref{SLemmatypes}.1 implies that 
\begin{packed_enum}
\item $v$ does not have neighbors $n \in B(R',T_i)$ and $n' \in B(R',T_j)$
  with $c(n)=c(n')=k$, and
\item $v$ does not have neighbors $m \in B(R',T_i)$ and $m' \in B(R',T_k)$
  with $c(m)=c(m')=j$.
\end{packed_enum}  
We claim that $v$ has a neighbor in
$n'' \in B(R')$ with $c(n'')=i$. Suppose not.
Then $c(N(v) \cap B(R',T_j))=k$, 
and therefore  $c(N(v) \cap B(R',T_i))=j$.
Also, $c(N(v) \cap B(R',T_k))=j$.
and therefore  $c(N(v) \cap B(R',T_i))=k$, a contradiction.
This proves that $c$ is a precoloring extension of  $D$.

Finally, we construct an augmented seeded precoloring $M_{U,c}$.
If $|L(D)(w)|=2$ for every $w \in W(D)$,
let $M_{U,c}=(D,\emptyset)$. Then $c$ is a coloring of $M_{U,c}$ and
$M_{U,c} \in \mathcal{M}(R)$, as required.
Thus we may assume that there exists $w \in \tilde{W}(D)$.
It follows from \eqref{smalltypes}, \eqref{smalltypes2} and
\ref{SLemmatypes}.1 that for every $\{i,j,k\}=\{1,2,3\}$,
and  every pair $T_i,T_j$ with $L(T_i)=i$ and $L(T_j)=j$, 
$w$ does not have neighbors $n \in B(R',T_i)$ and $n' \in B(R',T_j)$
with $c(n)=c(n')=k$.
For $l \in \{1,2,3\}$ let $W_l$ be the set of vertices $w \in \tilde{W}(D)$
 with neighbors in exactly $l$ of the sets
$B(D)_{ij}$.

\begin{equation}
  \label{mono}
  \longbox{Let $w \in W_2 \cup W_3$. We claim that
each of the  sets $N(w) \cap B(D)_{ij}$  is monochromatic in $c$.}
\end{equation}
Suppose not; we may assume that $1,2 \in c(N(w) \cap B(D)_{12})$.
It follows that $c(N(w) \cap (B(D)_{13} \cup B(D)_{23}))=3$,
and so  all three colors appear in $N(w)$, contrary to the fact that $c$ is
a precoloring extension of $D$.  This proves \eqref{mono}.

\bigskip

Suppose first that $W_3 \neq \emptyset$.
Then the three values $c(N(w) \cap B(D)_{ij})$ are all distinct,
again contrary to the fact that $c$ is a precoloring extension of $D$.
Thus we may assume that $W_3=\emptyset$.
Let $G'=G \setminus (W_1 \cup W_2)$, and let $\mathcal{X}$ be the set
off all the non-empty sets $N(w) \cap B(D)_{ij}$ with $w \in W_2$.
Let $M_{U,c}=((G',L(D), S(D)), \mathcal{X})$.
Then $M_{U,c} \in \mathcal{M}(D)$, and
by \eqref{mono} $c$ is a precoloring extension of $M_{U,c}$, as required.
This proves~\ref{stableto2SAT}.
\end{proof}

\section{The complete algorithm} \label{sec:complete}

We can now prove \ref{main} which we restate.
\begin{theorem} 
\label{main1}
The \textsc{list-3-coloring problem} can be
  solved in polynomial time for the class of $(P_6+rP_3)$-free graphs.
\end{theorem}

\begin{proof}
  The proof is by induction on $r$.  For $r=0$, the result follows from
  \ref{3colP7}, so we may assume that $r \geq 1$.
  Let $G$ be a $(P_6+rP_3)$-free graph and let
  $\tilde{L}$ be a $3$-list assignment  for $G$.
  We can test (by enumeration) if there exists $X \subseteq V(G)$
  with $|X| \leq 4r+8$ such that $(G|X,\tilde{L})$ is not
  colorable. If such $X$ exists, stop and output
  that $(G, \tilde{L})$ is not colorable.
  
  We may assume that $G$
  contains $P_6+(r-1)P_3$. Let $S \subseteq V(G)$ be such
  that  $G|S=P_6+(r-1)P_3$. For every precoloring 
  $(G, \tilde{L},S, L)$ of $(G,\tilde{L})$, $P_L=(G,L,S)$ is a
  $r$-seeded precoloring. Since $|S|=3r+6$, it follows that
  the number of such $r$-seeded precoloring is at most        
  $3^{3r+6}$. For each $P_L$ as above, let $\mathcal{L}_1(P_L)$ be
  be as in \ref{generaltonice}, and let
  $\mathcal{L}_1=\bigcup_{P_L}\mathcal{L}_1(P_L).$
  Then $|\mathcal{L}_1| \leq 3^{3r+6}|V(G)|^{g_1(3r+6)}$ and
  every member of $\mathcal{L}_1$ is nice or easy and has  seed of
  size at most $g_1(3r+6)$.

 For every $P' \in \mathcal{L}_1$ proceed as follows.
  If $P'$ is easy, set $\mathcal{L}(P')=\{P'\}$.
  Next assume that $P'$ is nice. Let $\mathcal{L}_2(P')$ be as
  in \ref{nicetostable}. Then $|\mathcal{L}_2(P')| \leq |V(G)|$,
  every member of
  $\mathcal{L}_2(P')$ is stable and has seed of size at most $g_1(3r+6)$.
  Now for every $P'' \in \mathcal{L}_2(P')$
  let $\mathcal{L}_3(P'')$ be as in \ref{stableto2SAT}. 
Then $|\mathcal{L}_3(P'')| \leq  |V(G)|^{g(g_1(3r+6))}$.
  Let $\mathcal{L}(P')=\bigcup_{P'' \in \mathcal{L}_2(P')}\mathcal{L}_3(P'').$
  Finally, let
  $\mathcal{L}=\bigcup_{P' \in \mathcal{L}_1}\mathcal{L}(P').$
  It follows that $|\mathcal{L}|$ is polynomial.
  
  It is now enough to test in polynomial time if each member of
  $\mathcal{L}$ has a precoloring extension.
  Let $Q \in \mathcal{L}$. It follows from the construction of
  $\mathcal{L}$ that $Q$ is either  a tractable augmented $r$-seeded
  precoloring, or an easy $r$-seeded precoloring.
If $Q$ is a tractable augmented $r$-seeded
precoloring, then a coloring of $Q$, or a determination that none exists,
can be found by \ref{Mono}.
  Thus we may assume that  $Q$ is an easy $r$-seeded precoloring.
  It is now enough to
  test if $(G\setminus X^0(L(Q)), L(Q))$ is colorable, and find a coloring if
  one exists. Since $V(G) \setminus X^0(L(Q)) \subseteq B(Q) \cup W(Q)$,
  this can be done by \ref{3colP7}. This proves~\ref{main1}.
  \end{proof}

\section{A hardness result} \label{sec:hardness}

A graph $G=(V,E)$ is said to be {\em k-critical} if $\chi(G)=k$ and $\chi(G-v)<k$
for any vertex $v\in V$. A $k$-critical graph $G$ is {\em nice} if $G$ contains three 
pairwise non-adjacent vertices $c_1$, $c_2$ and $c_3$ such that
$\omega(G-\{c_1,c_2,c_3\})=\omega(G)=k-1$.
For instance, any odd cycle of length at least seven with any
$3$-vertex stable contained in it 
is a nice 3-critical graph. The graph $H^*$ with its vertices $c_1$, $c_2$ and $c_3$ (see \autoref{fig:H^*})
 is a nice $4$-critical graph.
\begin{figure}[tb]
\centering
\begin{tikzpicture}[scale=0.6]
\tikzstyle{vertex}=[draw, circle, fill=white!100, minimum width=4pt,inner sep=2pt]

\node[vertex] (v1) at (0,0) {};
\node[vertex] (v2) at (1,2) {};
\node[vertex] (v3) at (-1,2) {};
\node[vertex] (c1) at (2,0) {$c_1$};
\node[vertex] (c2) at (0,4) {$c_2$};
\node[vertex] (c3) at (-2,0) {$c_3$};
\node[vertex] (u) at (0,1.5) {};
\draw (c1)--(v2)--(c2)--(v3)--(c3)--(v1)--(c1);
\draw (v1)--(v2)--(v3)--(v1);
\draw (u)--(c1);
\draw (u)--(c2);
\draw (u)--(c3);

\end{tikzpicture}
\caption{A nice $4$-critical graph $H^*$.}
\label{fig:H^*}
\end{figure}
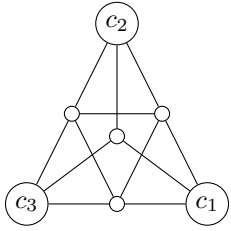

In \cite{huang}, the following generic framework of showing $NP$-completeness of the $k$-coloring problem
was proposed. Let $I$ be a \sat{3} instance with variables $x_1$, $x_2$, $\ldots$, $x_n$ and clauses
$C_1$, $C_2$, $\ldots$, $C_m$. Let $H$ be a nice $k$-critical graph. We construct a graph $G_{H,I}$
as follows.
\begin{framed}
\begin{packed_itemize}
\item For each variable $x_i$ there is a {\em variable component} $T_i$ consisting of
two adjacent vertices $x_i$ and $\overline{x_i}$. Call these vertices {\em $X$-type}.
\item For each variable $x_i$ there is a vertex $d_i$. Call these vertices {\em $D$-type}.
\item For each clause $C_j=y_{i_1}\vee y_{i_2}\vee y_{i_3}$ where $y_{i_t}$ is either $x_{i_t}$ or $\overline{x_{i_t}}$
there is a {\em clause component} $H_j$ that is isomorphic to $H$.
Denote the three specified pair-wise non-adjacent vertices in $H_j$ by
$c_{i_tj}$ for $t=1,2,3$.
Vertices $c_{i_tj}$ are referred to as {\em $C$-type} and all remaining vertices  in $H_j$ are referred to as {
\em $U$-type}.
\item Add an edge between every vertex of $U$-type and every vertex of $X$-type or $D$-type.
\item For each $C$-type vertex $c_{ij}$ we say that $x_i$ or $\overline{x_i}$ is its {\em literal vertex}
depending on whether $x_i\in C_j$ or $\overline{x_i}\in C_j$. Add an edge between $c_{ij}$ and
its literal vertex.
\item For each $C$-type vertex $c_{ij}$ add an edge between $c_{ij}$ and $d_i$.
\end{packed_itemize}
\end{framed}

\begin{lemma}[\cite{huang}]\label{lem:reduction}
A \sat{3} instance $I$ is satisfiable if and only if $G_{H,I}$ is $(k+1)$-colorable.
\end{lemma}

Now  we use the generic framework to prove \ref{P5P2} which we restate.
\begin{theorem}
\label{P5P22}
The \textsc{$k$-coloring problem} restricted to
$P_5+P_2$-free graphs is $NP$-hard for $k \geq 5$.
\end{theorem}

\begin{proof}

  First we show:
  
  \begin{equation}
    \label{lem:p5p2free}
\longbox{Let $I$ be a \sat{3} instance and $H$ be a nice $k$-critical graph.
If $H$ is $P_5$-free, then $G_{H,I}$ is $(P_5+P_2)$-free.}
\end{equation}
Suppose that $G_{H,I}$ contains
an induced $Q=Q_1+Q_2$ where $Q_1$ and $Q_2$ are isomorphic
to a $P_5$ and a $P_2$, respectively. 
Let $C_i$ (respectively $\overline{C_i}$) be
the set of $C$-type vertices that connect to $x_i$ (respectively $\overline{x_i}$).
We observe that each connected component of $G-U$ has a specific
structure, namely it is the result of substituting stable sets into a $5$-cycle
(and possibly removing some vertices).  Specifically, the $5$ stable sets are,
in the cyclical order, $X_0=\{x_i\}$, $X_1=C_i$, $X_2=\{d_i\}$,
$X_3=\overline{C_i}$, and finally $X_4=\{\overline{x_i}\}$. This subgraph
does not contain an induced $P_5$, since the 5-cycle does not and substituting stable
sets cannot create a $P_5$. 
This implies that $Q_1\cap U\neq \emptyset$. 
Since $U$ is complete to $X\cup D$, $Q_2\subseteq U\cup C$.
Since $C$ is an stable set, this implies that $Q_2\cap U\neq \emptyset$ and thus
$Q_1\subseteq U\cup C$. This means that $Q_1$ is entirely contained in some clause component.
This, however, contradicts the assumption that $H$ is $P_5$-free.
This proves \eqref{lem:p5p2free}.

\bigskip

Now observe that the graph $H^*$ (\autoref{fig:H^*}) is $P_5$-free. It follows
then from \autoref{lem:reduction} and \eqref{lem:p5p2free} that 
\col{5} $(P_5+P_2)$-free graphs is $NP$-hard.
\end{proof}

\end{document}